\documentclass{article}
\usepackage{graphicx, caption} 

\usepackage{amsmath}
\usepackage{amsthm}
\usepackage{amsfonts}
\usepackage{mathabx}
\usepackage{geometry}
\usepackage{bbm}
\usepackage{multirow}
\usepackage{float}
\usepackage[colorlinks]{hyperref}
\usepackage{cleveref}
\usepackage{biblatex} 
\usepackage[export]{adjustbox}

\usepackage{subcaption}

\DeclareMathOperator*{\argmax}{arg\,max}
\DeclareMathOperator*{\argmin}{arg\,min}
\DeclareMathOperator*{\supp}{supp}
\DeclareMathOperator*{\Span}{span}

\newcommand{\N}{\mathbb{N}}
\newcommand{\Z}{\mathbb{Z}}
\newcommand{\C}{\mathbb{C}}
\newcommand{\R}{\mathbb{R}}
\newcommand{\cH}{\mathcal{H}}

\theoremstyle{definition}

\newtheorem{definition}{Definition}

\theoremstyle{remark}
\newtheorem{remark}{Remark}

\theoremstyle{plain}

\newtheorem{proposition}{Proposition}

\newtheorem{theorem}{Theorem}
\newtheorem{lemma}{Lemma}

\theoremstyle{plain}

\usepackage{algorithm}
\usepackage{algpseudocode}

\hypersetup{
    colorlinks=true,
    linkcolor=blue,
    filecolor=magenta,      
    urlcolor=cyan,
    pdftitle={On the sample complexity of compressed sensing with wavelets and shearlets},
    pdfpagemode=FullScreen,
    }

\title{On the sample complexity of Fourier compressed sensing: wavelets versus shearlets}
\author{Giovanni S. Alberti\footnotemark[1]\thanks{Machine Learning Genoa Center (MaLGa), Department of Mathematics, Department of Excellence
2023-2027, University of Genoa, Italy (\href{mailto:giovanni.alberti@unige.it}{giovanni.alberti@unige.it}, \href{mailto:ayse.isil.guleken@edu.unige.it}{ayse.isil.guleken@edu.unige.it}).} \and Alessandro Felisi\footnotemark[2]\thanks{Institute for Applied Mathematics (IAM), University of Bonn, Germany (\href{mailto:afelisi@uni-bonn.de}{afelisi@uni-bonn.de}).} \and Işıl Güleken\footnotemark[1]}
\date{July 22, 2026}

\begin{document}

\maketitle

\begin{abstract}
This paper explores the measurement requirements for signal recovery in compressed sensing, comparing the performance of shearlet frames with traditional wavelet systems. Directional representation systems such as shearlets are known for their ability to sparsely represent images with anisotropic features, which allows for efficient nonlinear approximations. The central question we address is whether this difference in sparsity allows for a proportional reduction in the number of Fourier measurements needed for successful reconstruction.

On the theoretical front, we study the obstacles encountered when trying to apply standard frame-based recovery results to shearlet systems. First, we show that the (optimal) local coherence between Fourier measurements and cone-adapted shearlets decays more slowly than the corresponding local coherence for wavelets. Second, managing the sparsifying system's redundancy relies on evaluating a localization factor, which requires lower frame bound estimates that can become exponentially small. Thus, even under optimal theoretical conditions, the number of samples required for shearlets scales quadratically with  sparsity, which offers no substantial theoretical reduction over the standard wavelet benchmark.

These theoretical limitations are assessed through a series of numerical experiments on a dataset of piecewise smooth images. While empirical observations confirm that shearlets can accurately represent these images using fewer coefficients than wavelets, phase diagrams indicate that this advantage in sparsity does not yield a proportional reduction in required Fourier samples. Ultimately, we conclude that despite the superior nonlinear approximation rates of shearlets, their practical sample complexity in compressed sensing scenarios with subsampled Fourier measurements remains comparable to that of traditional wavelets.
\end{abstract}

\section{Introduction}

\subsection{Sampling and compressed sensing}

The reconstruction of signals has traditionally relied on the Shannon-Nyquist sampling theorem, which states that a signal must be sampled at a rate at least twice its highest frequency to guarantee exact recovery \cite{Nyquist1928, Shannon1949}. While this principle remains central to digital signal processing, it often requires demanding sampling rates for broadband signals, leading to challenges in data acquisition and storage. Over the past two decades, the framework of compressed sensing has shown that signals can be reliably recovered from a number of measurements well below the conventional Nyquist threshold \cite{Candes2006, Donoho2006}. The most important assumption enabling this reduction in sample complexity is sparsity: the underlying signal can be well-approximated by a small number of non-zero coefficients in a suitable representation system. By leveraging this structural assumption, the recovery of a signal from partial, undersampled measurements is formulated as an underdetermined inverse problem, which can be solved using $\ell^1$-minimization.

The theoretical guarantees for compressed sensing depend fundamentally on the properties of the measurement operator. Initial breakthroughs mostly analyzed dense random matrices, such as those drawn from Gaussian distributions \cite{CandesTao2006}. These matrices exhibit near-optimal mathematical properties, the most notable being the restricted isometry property, which ensures robust recovery of sparse signals with high probability. However, random Gaussian measurements are often not physically realizable. As a result, a vast body of literature has focused on structured sensing modalities, particularly subsampled Fourier measurements. This measurement model accurately reflects practical imaging architectures, such as magnetic resonance imaging (MRI) and radio interferometry \cite{Lustig2007}. In these settings, successful recovery typically relies on analyzing the mutual coherence between the measurement system and the chosen sparsifying frame \cite{CandesRomberg2007, Foucart2013}.

\subsection{Compressed sensing with wavelets and shearlets}\label{sub:CS-intro}

The effectiveness of recovery depends strongly on the choice of the sparsifying basis or frame, which must accurately reflect the structural properties of the signal. In the simplest theoretical scenarios, signals may exhibit sparsity directly in the canonical basis. This setting is highly tractable; for instance, the canonical and Fourier bases are maximally incoherent, providing ideal conditions for uniform recovery guarantees \cite{Foucart2013}. However, complex natural data, such as real-world images, are rarely sparse in the spatial domain and require more sophisticated representation systems to concentrate their energy into a few dominant coefficients. For decades, wavelets have been the standard choice for image processing, owing to their ability to sparsely represent piecewise smooth functions and isolate isotropic point singularities \cite{Mallat1998}. As a result, wavelet systems have been extensively used in compressed sensing, forming the foundation for many standard recovery architectures in imaging. Moreover, when these representations, whether the canonical basis or suitable wavelet systems, are paired with structured measurements like subsampled Fourier measurements, the sample complexity is theoretically optimal. Rigorous analyses demonstrate that under appropriate variable-density sampling strategies, the required number of measurements $m$ scales linearly with the signal's sparsity $s_{\rm wave}$, modulo logarithmic factors \cite{RudelsonVershynin2008, KrahmerWard2014}, namely, 
\begin{equation}\label{eq:sample-complexity-wavelets}
m\gtrsim s_{\rm wave}   . 
\end{equation}

While wavelets are highly effective for isotropic structures, they are suboptimal for capturing highly anisotropic features, such as elongated edges in natural images. To overcome this limitation, directional representation systems were introduced, with curvelets \cite{CandesDonoho2004} and shearlets \cite{Labate2005, Guo2007} emerging as the most prominent frameworks. For the standard mathematical model of \emph{cartoon-like images}—functions that are smooth except along $C^2$ singularity curves—these directional systems offer superior nonlinear approximation properties. Formally, let $f_N$ denote the $N$-term approximation of an image $f$, formed by using only the $N$ largest frame coefficients in absolute value. The nonlinear approximation error is defined as the squared $L^2$-norm of the residual, $\|f - f_N\|_2^2$. For cartoon-like images, this error decays at a rate of $\mathcal{O}(N^{-1})$ for 2D wavelets \cite{Mallat1998}, whereas shearlets and curvelets achieve an essentially optimal decay rate of $\mathcal{O}(N^{-2} (\log N)^3)$ (see \cite{CandesDonoho2004, Guo2007} and Appendix~\ref{shearletAppendix}). This faster decay implies that the effective sparsity $s_{\rm shear}$ of such an image with respect to a shearlet frame is considerably smaller than the sparsity $s_{\rm wave}$ with respect to a wavelet basis. More precisely, ignoring log factors, the above bounds imply 
\begin{equation}\label{eq:swaveshear}
    s_{\rm wave} \asymp s_{\rm shear}^2,
    \end{equation}
    for a prescribed approximation accuracy. 
    This raises the natural question of whether this improved sparsity directly yields proportional improvements in the sample complexity bounds for compressed sensing scenarios, namely,
\begin{equation}\label{eq:sample-complexity-shearlets}
m\gtrsim s_{\rm shear}.
\end{equation}

To investigate this issue, a growing body of literature has integrated shearlets into compressed sensing frameworks, considering either shearlets explicitly \cite{Ma2016thesis, Ma2017} or general frames \cite{rauhut-etal-2008,Candes2011,Krahmer2015,Poon2017, Alberti2021}. However, to the best of our knowledge, rigorous and explicit sample complexity estimates for shearlet systems, relating $m$ and $s_{\rm shear}$ as in \eqref{eq:sample-complexity-shearlets}, are not currently available. Shearlets do not form an orthonormal basis, so classical compressed sensing results for orthonormal bases are not applicable. Existing theoretical studies provide only partial recovery guarantees, as they rely on structural assumptions, coherence bounds, localization and balancing properties that are  difficult to quantify for shearlets. Key partial results are obtained in \cite{Ma2016thesis}, but are not conclusive. A closely related positive result was obtained in \cite{kutyniok-lim-2018}, where the authors prove asymptotic almost optimality for the recovery of cartoon-like functions using dualizable shearlets, anisotropic directional sampling, and a tailored sequence of $\ell^1$ reconstructions. This result is complementary to our setting: it relies on a restricted cartoon-like class, including a necessary non-vanishing curvature condition for the discontinuity curve, as well as on a special shearlet system and a problem-specific sampling/reconstruction procedure, rather than providing an explicit \(m\)-versus-\(s_{\rm shear}\) sample-complexity estimate for the standard shearlet--Fourier compressed sensing framework.

Despite this theoretical gap, empirical evidence seems to support the expected advantages of directional systems. In practical scenarios, such as accelerated MRI, some works suggest that shearlet regularization achieves higher reconstruction quality from  undersampled Fourier data compared to standard wavelet approaches \cite{Ma2018, Yazdanpanah2017}.

\subsection{Main contributions}

To systematically understand this theoretical bottleneck, the first part of our paper investigates the specific mathematical obstacles that arise when applying established compressed sensing recovery theorems for general frames to shearlet systems with subsampled Fourier measurements. The local coherence between cone-adapted shearlets and Fourier measurements can be adequately bounded as shown in \Cref{subsecLocalCoherence}, yielding an (optimal) estimate that is worse than the corresponding bound for wavelets. Controlling the redundancy of the sparsifying frame remains unfeasible using existing mathematical methods, as discussed in \Cref{subsecLocalizationFactor}. Specifically, existing recovery guarantees rely heavily on the localization factor of the frame. We show that bounding this factor requires estimating the optimal lower frame bounds for arbitrary finite subframes of shearlets, and these bounds may be exponentially small. 
However, even in the best-case scenario when the localization factor is of order one, for cartoon-like images the coherence estimates derived in \Cref{subsecLocalCoherence} yield a sample complexity bound of the form
\[
m\gtrsim s_{\rm shear}^2,
\]
see \Cref{sub:comparing}. Surprisingly, this bound does not give the desired sample complexity \eqref{eq:sample-complexity-shearlets} and, in view of \eqref{eq:swaveshear}, is comparable to \eqref{eq:sample-complexity-wavelets} for cartoon-like images.

To complement this theoretical analysis, \Cref{sectionNumerics} presents a comprehensive numerical investigation into the practical performance of these representation systems. We first empirically confirm the expected nonlinear approximation behavior: for a generated dataset of piecewise regular, cartoon-like images, the number of coefficients required for an accurate representation is substantially smaller using a shearlet transform compared to a standard wavelet transform. However, our main numerical finding is that this significant advantage in sparsity does not translate into a proportional reduction in sample complexity. Through a series of phase diagrams, we systematically compare the recovery success rates of both sparsifying frames under varying Fourier undersampling ratios and varying levels of signal complexity. These experiments reveal that the number of measurements required to successfully recover the same class of images is  comparable for both wavelets and shearlets. Ultimately, our numerical results demonstrate that the superior approximation rates of directional systems do not automatically guarantee a proportional improvement in Fourier sample complexity. This is consistent with the theoretical analysis of the problem described in the previous paragraph.

\subsection{Structure of the paper}

The remainder of this paper is organized as follows. \Cref{sec:frames} establishes the mathematical foundation by reviewing the necessary concepts from frame theory and the abstract formulation of the compressed sensing recovery problem. In \Cref{sectionShearletUnknowns}, we introduce the continuous and discrete shearlet frameworks and detail the theoretical obstacles encountered when attempting to bound the required coherence and localization factors. Finally, \Cref{sectionNumerics} is dedicated to our numerical experiments, presenting the phase diagrams that systematically compare the empirical sample complexity of wavelets and shearlets. Additional technical details, including formal definitions of parabolic molecules and shearlets, are deferred to the appendices.

\section{Compressed sensing with frames}\label{sec:frames}

Let $\mathcal{H}$ be a separable complex Hilbert space, representing our signal space, which may be finite or infinite dimensional. The main problem of compressed sensing is the recovery of an unknown signal $g^\dagger \in \mathcal{H}$ from partial measurements of the form $(\langle g^\dagger, \varphi_l\rangle_\mathcal{H})_l$, under a sparsity assumption on $g^\dagger$ with respect to a suitable family of vectors $\{\psi_i\}_i$. A common assumption made on the two sequences $\{\varphi_l\}_l$ and $\{\psi_i\}_i$ (as in \cite{Poon2017,Alberti2021}), which we will later refine based on our situation, is that they both form frames.

\subsection{Frame theory setting}\label{sec:frametheory}

We provide here a short introduction to the concepts we need from frame theory. For further details, the reader is referred to \cite{Christensen2016}. We say that a subset of $\mathbb{N}$ is an \emph{index set} if it is either $\mathbb{N}$ or $\{1, \dots, n\}$ for some $n \in \mathbb{N}$. 

\begin{definition}
Let $I \subseteq \mathbb{N}$ be an index set.

\begin{itemize}
\item We say that a family $\{f_i\}_{i \in I} \subseteq \mathcal{H}$ is a \emph{frame} for $\mathcal{H}$ if there exist real constants $A,B > 0$ such that
$$
A \|f\|^2 \leq \sum_{i \in I} |\langle f, f_i \rangle|^2 \leq B \|f\|^2,
\quad \text{for all } f \in \mathcal{H}.
$$
If $A=B$, we say that the frame is \emph{tight}. If $A=B=1$, we say that the frame is \emph{Parseval} (because the above becomes Parseval's identity). 

\item Given a frame $\{f_i\}_{i \in I}$, the \emph{analysis operator} $T \colon \mathcal{H} \to \ell^2(I)$ is defined by
$$
Tf = \{\langle f, f_i \rangle\}_{i \in I}, \quad f \in \mathcal{H}.
$$

\item The \emph{synthesis operator} $T^* \colon \ell^2(I) \to \mathcal{H}$ associated to the frame $\{f_i\}_{i \in I}$ is the adjoint of $T$, given by
$$
T^* c = \sum_{i \in I} c_i f_i, \quad c = \{c_i\}_{i \in I} \in \ell^2(I).
$$

\item The \emph{frame operator} $S \colon \mathcal{H} \to \mathcal{H}$ associated to the frame $\{f_i\}_{i \in I}$ is defined by $S=T^* T$, namely,
$$
Sf = T^* T f = \sum_{i \in I} \langle f, f_i \rangle f_i, \quad f \in \mathcal{H}.
$$
\end{itemize}
\end{definition}

Let $L$ and $I$ be index sets. Consider two frames $\{\varphi_l\}_{l \in L}$ and $\{\psi_i\}_{i \in I}$ for $\mathcal{H}$, with analysis operators $V$ and $D$, respectively. Let their optimal frame bounds be given by $B_V \geq A_V > 0$ and $B_D \geq A_D > 0$.

The \emph{dual frames} of $\{\varphi_l\}_{l \in L}$ and $\{\psi_i\}_{i \in I}$, denoted by $\{\Tilde{\varphi}_l\}_{l \in L}$ and $\{\Tilde{\psi}_i\}_{i \in I}$, are given by 
$$\Tilde{\varphi}_l = (V^*V)^{-1}\varphi_l, \,\,\, \text{ and } \,\,\, \Tilde{\psi}_i = (D^*D)^{-1}\psi_i.$$
The \emph{Moore-Penrose pseudoinverses} of $V$ and $D$ are 
$$V^{-1} = (V^*V)^{-1}V^*, \,\,\, \text{ and } \,\,\, D^{-1} = (D^*D)^{-1}D^*,$$
which are the left inverses of $V$ and $D$, respectively. 
Later on we will need the following result that relates the eigenvalues of a frame operator to the optimal frame bounds.

\begin{proposition}[{\cite[Theorem 1.3.1 (i)]{Christensen2016}}] \label{frameBound}
        Let $\{f_i\}_{i=1}^N$ be a finite frame for a finite-dimensional Hilbert space $\mathcal H$, with frame operator $S$. Then the optimal lower frame bound is the smallest eigenvalue of $S$ and the optimal upper frame bound is the largest eigenvalue of $S$.
\end{proposition}

\subsection{Compressed sensing setting}
\label{sec:compressed_sensing_setting}

The recovery problem can be stated as follows. We consider noisy partial measurements of $V g^\dagger$ of the form
$$
y_\Omega = P_\Omega V g^\dagger + \gamma_\Omega,
$$
where $\Omega$ denotes the sampled measurement indices, $P_\Omega$ denotes the orthogonal projection onto $\Span\{e_l : l \in \Omega\}$,  $\{e_l\}_{l}$ is the canonical basis of $\ell^2$ and $\gamma$ represents measurement noise with $\|\gamma_\Omega\|_2 \leq \varepsilon_\Omega$ for some noise level $\varepsilon_\Omega \geq 0$. If $\Omega = \{1, \dots, N\}$, we simply denote it by $P_N$. Our aim is to recover the
signal $g^\dagger \in \mathcal{H}$, under the assumption that $D g^\dagger$ is sparse (see below for a precise definition). The classical way to solve the sparse recovery problem is the following $\ell^1$ minimization problem:
\begin{equation}
\inf_{\substack{g \in \mathcal{H}, \\ Dg \in \ell^1(I)}} \|Dg\|_1 \quad \text{ subject to } \quad \|P_\Omega Vg - y_\Omega\|_2 \leq \varepsilon_\Omega. \label{analysisFormulation}  
\end{equation}
This is the so-called \emph{analysis formulation}, in which the analysis operator $D$ appears directly in the $\ell^1$ norm.

We will present the recovery result from \cite{Krahmer2015}. For the rest of this section, we assume that $D$ is Parseval. We first introduce the quantities appearing in the theorem. We start by defining the local coherence, which modifies the classical coherence by adding a dependence on $l\in L$.

\begin{definition}[Coherence] The \emph{local coherence} of a sequence $\boldsymbol{\psi} = \{\psi_i\}_{i\in I}$ with respect to another sequence $\boldsymbol{\varphi} = \{\varphi_l\}_{l \in L}$ is the function $\mu^{\rm loc}(\boldsymbol{\psi}, \boldsymbol{\varphi})$ defined coordinate-wise as

$$\mu^{\rm loc}_l(\boldsymbol{\psi}, \boldsymbol{\varphi}) = \mu^{\rm loc}_l = \sup_{i \in I} |\langle \psi_i, \varphi_l \rangle|.$$

\end{definition}

As anticipated above, our construction is based on the assumption that $x=Dg^\dagger$ is sparse. 
\begin{definition}[Sparsity]
Let $N\in\mathbb{N}$. Let $x\in\mathbb{C}^N$ or $x\in\ell^2(\mathbb{N})$. Let $s \in \mathbb{N}$ with $s \leq N$.  
We say that $x$ is $(s,N)$-\emph{sparse} if $\supp(x) \subseteq \{1, \dots, N\}$ and $|\supp(x)| \leq s$ and denote the set of $(s,N)$-sparse vectors by $\Sigma_{s,N}$. We use the notation
$$
\sigma_{s,N}(x) = \min_{x_0 \in \Sigma_{s,N}} \{\|x - x_0\|_1\}.
$$
\end{definition}
In practice, $x=Dg^\dagger$ may not be exactly sparse, but compressible, in the sense that the approximation error $\sigma_{s,N}(x)$ is small.

Next, we introduce the localization factor.

\begin{definition}\label{def:localization_factor}
    We define the \emph{localization factor} $\eta_{s, D}$ as
    $$\eta_{s, D} = \sup\bigg\{\frac{\|Dg\|_1}{\sqrt{s}} : |\Delta| = s, \, g\in \operatorname{Im}(D^*P_\Delta), \, \|g\|_{\mathcal{H}} = 1\bigg\}.$$
\end{definition}

This quantity measures the redundancy of the sparsifying system $\{\psi_i\}_{i\in I}$. Indeed, if we have an orthonormal basis, for which $D$ is unitary, by Cauchy-Schwarz we have $\eta_{s, D}=1$ \cite{Krahmer2015}. However, $\eta_{s, D}$ may be arbitrarily large with highly redundant frames.

\subsection{An abstract recovery result}

A numerical implementation of the analysis formulation \eqref{analysisFormulation} is not possible when $I$ is infinite. Thus, for the numerical experiments with wavelet and shearlet sparsifying frames in \Cref{sectionNumerics}, we restrict to the case where the index sets for the sparsifying and measurement frames are $I = \{1, \cdots, N\}$ and $L = \{1, \cdots, M\}$ respectively. We also consider the case where the  sparsifying frame is Parseval and the measurement frame is an orthonormal basis. This covers the setting with sparsity modeled by wavelets or bandlimited shearlets and with Fourier measurements.

\begin{theorem}[{\cite[Corollary 2.9]{Krahmer2015}}]\label{corollary2.9}
Let $N, s \in \mathbb{N}$ be such that $s < N$. Let $D\in\mathbb{C}^{N\times M}$ with rows $\{\psi_1,\ldots,\psi_N\}$ be a Parseval frame and let $V \in \mathbb{C}^{M \times M}$ with rows $\{\varphi_1,\ldots,\varphi_M\}$ be an orthonormal basis of $\mathbb{C}^M$. Let $\omega \in \mathbb{R}^M_+$ be such that
\begin{equation}\label{eq:muk-omegak}
    \mu_k^{\rm loc}(\{\psi_i\}_{i = 1}^N, \{\varphi_l\}_{l = 1}^M) \leq \omega_k, \quad 1 \leq k \leq M.
\end{equation}

Set the diagonal matrix $W = \mathrm{diag}(\mathbf{w})\in\mathbb{C}^{M\times M}$ with $\mathbf{w}_l=\|\omega\|_2/\omega_l$. Assume that the following holds:

\begin{enumerate}
    \item let 
    \begin{equation}\label{eq:cor:m}
    m \ge C_1\eta_{s, D}^2\|\omega\|_2^2s\log^3(s\eta_{s, D}^2)\log(N)
    \end{equation}
    where $C_1 > 0$ is a universal constant;

    \item let $m$ indices $\{l_1, \dots, l_m\}$ be sampled independently from $\{1, \dots, M\}$ according to the probability distribution
    $\nu$ on $\{1, \dots, M\}$ given by 
    \begin{equation}\label{eq:nul-omegal}
        \nu_l=\frac{\omega_l^2}{\|\omega\|_2^2}, \quad 1 \leq l \leq M,
    \end{equation} and set $\Omega = \{l_1, \dots, l_m\}$ (with possible repetitions to be kept);

    \item take $g^\dagger \in \mathbb{C}^M$ and noise $\gamma \in \mathbb{C}^M$ with weighted error $\left\|\frac{1}{\sqrt{m}}W\gamma\right\|_2\leq \varepsilon$ and set $\zeta = P_\Omega Vg^\dagger+\gamma$. Let $g \in \mathbb{C}^M$ be a solution of the weighted $\ell^1$-analysis problem
    \begin{equation}\label{eq:weighted-noise}\argmin_{g\in\mathbb{C}^M}\|Dg\|_1
    \quad\text{subject to}\quad
    \left\|\frac{1}{\sqrt{m}}W(P_\Omega Vg-\zeta)\right\|_2 \leq \varepsilon.
     \end{equation}
\end{enumerate}

Then, with probability $1-N^{-\log^3(2s)}$, we have
$$
\|g - g^\dagger\|_2 \leq C_2\frac{\sigma_{s, N}(Dg^\dagger)}{\sqrt{s}} + C_3\varepsilon,
$$
where $C_2, C_3 > 0$ are universal constants.
\end{theorem}

As discussed in \cite{Krahmer2015}, the weighted noise model with the diagonal matrix $W$ in \eqref{eq:weighted-noise} may be an artifact of the proof technique. Analogous results with an unweighted noise model are present in \cite{alberti2025compressedsensinginverseproblems}. Thus, in the numerical simulations, we will consider the simpler unweighted model \eqref{analysisFormulation}.

Related recovery results are available in the infinite-dimensional setting, see \cite{Alberti2021, Poon2017}. In particular, in   \cite[Corollary 1]{Alberti2021} the authors consider arbitrary frames (not necessarily Parseval) for both measurement and sparsifying systems. Under more general restrictions on the frames, the results in \cite{Poon2017} consider the framework of multilevel sampling and multilevel sparsity, originally introduced in \cite{Poon2017-2}. Both these infinite-dimensional approaches involve a \textit{balancing property}, which quantifies how large the measurement truncation must be. Furthermore, additional constants that depend on the redundancy of the sparsifying frame need to be estimated to obtain a lower bound on the number of samples required. See   Appendix~\ref{metaResultAppendix} for additional details on these infinite-dimensional results.

\begin{remark}\label{remarkUnknownQuantities}
    The number of samples $m$ needed to satisfy both the assumptions of \Cref{corollary2.9} and the infinite-dimensional results (see Appendix~\ref{metaResultAppendix}) and their recovery-error bounds depend on the following quantities:
    \begin{itemize}
    \item the sparsity $s$ of the unknown $g^\dagger$ with respect to $D$, through the approximation error $\sigma_{s,N}(Dg^\dagger)$;
    \item  the local coherence between the sparsifying and the measurement frames $\mu^{\rm loc}(\{\psi_i\}_{i\in I}, \{\varphi_l\}_{l\in L})$;
    \item the localization factor $\eta_{s, D}$, measuring the redundancy of the sparsifying frame.

    \end{itemize}
    
\end{remark}

It has been shown in \cite{Guo2007} that shearlets perform better than wavelets in representing a commonly used class of functions called \emph{cartoon-like images} (\cite{CandesDonoho2004}, see Appendix~\ref{shearletAppendix}), as anticipated in the Introduction. In other words, for this class of functions, the quantity $\sigma_{s,N}$ is expected to be smaller for shearlets than for wavelets. Because of the other two factors $\eta_{s,D}$ and $\mu^{\rm loc}$, it is unclear from this result whether the sample complexity estimate given in \eqref{eq:cor:m} is better for wavelets or shearlets. In the following two sections, we address this question from both theoretical and numerical viewpoints, respectively.

\section{Obstacles to recovery results for shearlets} \label{sectionShearletUnknowns}
In this section, we study how \Cref{corollary2.9} could be applied when shearlets are used as the sparsifying frame. More precisely, we shall try to estimate the local coherence $\mu^{\rm loc}$ and the localization factor $\eta_{s,D}$. 

\subsection{A brief introduction to shearlets}
\label{sec:introduction_shearlets}

First, we provide a brief introduction to shearlets. For additional details, see Appendix~\ref{shearletAppendix}.

We are interested in the case where the sparsifying frame $\{\psi_i\}_{i\in I}$ of $L^2(\mathbb{R}^2)$ consists of wavelet-like functions, allowing us to compare their performance against standard wavelets. One such class of functions, called \emph{shearlets}, is a representation system used to encode highly anisotropic features, originally introduced in \cite{Labate2005}. In many imaging applications, important structures, such as edges in natural images, are typically elongated and oriented along curves. While classical wavelets are well-suited to capture isotropic structures, they are less efficient at representing these directional phenomena. Shearlets were introduced to address this limitation by providing a multiscale system whose elements can adapt to both the scale and the orientation of anisotropic features.

The main idea is to construct shearlets from a single generating function using three types of operations, which are parabolic scaling, shearing and translation. In the continuous setting, a shearlet system is obtained by applying these operations to a mother shearlet $\psi\in L^2(\mathbb{R}^2)$, producing functions that become increasingly elongated at finer scales while also varying their orientation. This construction allows shearlets to efficiently capture geometric structures such as edges aligned along curves. More precisely, shearlets are given by
$$
\psi_{a,s,t}(x) = a^{-3/4}\psi(D_a^{-1}B_s^{-1}(x-t)),
$$
where $a > 0$ is the scale,  $s \in \mathbb{R}$ is the shearing parameter representing the orientation, $t \in \mathbb{R}^2$ is the location in space, and the parabolic scaling and shearing matrices are given by
$$
D_{a} = \begin{pmatrix}
a & 0 \\
0 & \sqrt{a} 
\end{pmatrix}, \quad B_s = \begin{pmatrix}
1 & s \\
0 & 1 \end{pmatrix},
$$
respectively. By varying these parameters we obtain a rich family of directional elements capable of capturing anisotropic features across multiple scales. For an accessible introduction to the continuous shearlet framework, we refer to \cite{Kutyniok2012}.

Parabolic scaling means that one spatial direction is refined faster than the other, producing elements whose effective support obeys a relationship of the form $\text{width} \approx \text{length}^2$. In what is called the \emph{cone-adapted} setting, the frequency plane is separated into two directional cones, corresponding roughly to predominantly horizontal and predominantly vertical directions, together with a low-frequency region around the origin. The directional shearlet elements are used to analyze the two cones, while the low-frequency part is handled separately by a scaling function. More specifically, shearlets associated to different cones are scaled parabolically due to the matrices $A_{(a, 1)}$ and $A_{(a, 2)}$ and they both display a form of directionality due to the shearing matrices $B_s$. These dilation matrices are defined by
$$
A_{(a, 1)} = \begin{pmatrix}
a & 0 \\
0 & \sqrt{a} 
\end{pmatrix}, \quad A_{(a, 2)} = \begin{pmatrix}
\sqrt{a} & 0 \\
0 & a \end{pmatrix}.
$$

Shearlet parameters are usually discretized by sampling the parameter set of continuous shearlet systems. In such a discretization, the continuous scale, shearing and translation parameters are replaced by countable subsets. The scale is indexed by an integer $j$, the shear parameter by an integer $\ell$ depending on the scale and the location by a lattice point $k\in\mathbb{Z}^2$. Thus, the discrete system can be viewed as a structured sampling of the continuous family, chosen so that the resulting elements still cover different scales, orientations and spatial positions. To define one commonly used such discretization, we use the following notation for the sampling matrices, 
$$
M_{(c, 1)} = \begin{pmatrix}
c_1 & 0 \\
0 & c_2
\end{pmatrix}, \quad M_{(c, 2)} = \begin{pmatrix}
c_2 & 0 \\
0 & c_1
\end{pmatrix}, \quad c = (c_1, c_2) \in (\mathbb{R}^+)^2.
$$

\begin{definition}[{\cite[Definition 2.2]{Kutyniok2012-2}}] \label{shearletsInSpace}
    For some sampling vector $c=(c_1,c_2)\in(\mathbb{R}^+)^2$, the \emph{(cone-adapted) regular discrete shearlet system}
    $\mathcal{SH}(c;\Phi,\psi^{(1)},\psi^{(2)})$ generated by a scaling function $\Phi\in L^2(\mathbb{R}^2)$ and horizontal and vertical mother
    shearlets $\psi^{(1)},\psi^{(2)}\in L^2(\mathbb{R}^2)$ respectively is defined by
    $$
    \mathcal{SH}(c;\Phi,\psi^{(1)},\psi^{(2)})
    = \Psi_0(c,\Phi)\cup \Psi_1(c,\psi^{(1)})\cup \Psi_2(c,\psi^{(2)}),
    $$
    where
    $$
    \Psi_0(c,\Phi)=\{\psi_{-1,k}=\Phi(\,\cdot - c_1 k): k\in\mathbb{Z}^2\},
    $$
    $$
    \Psi_1(c,\psi^{(1)})=\{\psi^{(1)}_{j,\ell,k}=2^{3j/4}\psi^{(1)}(B_{-\ell} A_{(2^j, 1)}\,\cdot - M_{(c, 1)} k):
    j \geq 0,\ |\ell| \leq \lceil 2^{j/2} \rceil,\ k\in\mathbb{Z}^2\},
    $$
    and
    $$
    \Psi_2(c,\psi^{(2)})=\{\psi^{(2)}_{j,\ell,k}=2^{3j/4}\psi^{(2)}(B_{-\ell}^T A_{(2^j, 2)}\,\cdot - M_{(c, 2)} k):
    j\geq 0,\ |\ell| \leq \lceil 2^{j/2} \rceil,\ k\in\mathbb{Z}^2\}.
    $$
\end{definition}

\begin{remark}\label{rem:c1c2}
    For the digital implementation of cone-adapted shearlets, which will be relevant for our numerical experiments in \Cref{sectionNumerics}, the finite dimensionality of the compressed sensing problem naturally gives rise to a shearlet reconstruction system with only finitely many elements. The common choice of subsampling matrices for such a setting is such that the resulting finite shearlet system becomes highly redundant. More specifically, for a shearlet system up to a certain finest scale $j_0 \in \mathbb{Z}$, the sampling matrices $M_{(c, 1)} = \text{diag}(c_1, c_2)$ and $M_{(c, 2)} = \text{diag}(c_2, c_1)$ are chosen to depend on scale such that $c_1 = c_1^j = 2^{j - j_0}$ and $c_2 = 2^{j/2 - j_0}$ (see \cite[Section 3.4]{Kutyniok2016}). This results in the relevant digital shearlet transform to produce an equal number of shearlet coefficients for each feasible scale and shearing parameter pair, instead of the number of shearlet coefficients varying for different scales $j$ where $0 \leq j \leq j_0$ \cite{Kutyniok2016}. For the theoretical investigation that considers both finite and infinite frames at once in  \Cref{subsecLocalCoherence,subsecLocalizationFactor} however, the subsampling matrices are assumed to be simply $M_{(c, 1)} = M_{(c, 2)} = I_2$.
\end{remark}

The importance of shearlets for the present work lies in their remarkable nonlinear approximation properties for images with anisotropic singularities. As mentioned in the Introduction and discussed in Appendix~\ref{shearletAppendix}, if $f_N$ denotes the best $N$-term approximation of a cartoon-like image $f$ obtained by retaining the $N$ largest shearlet (or wavelet) coefficients (see \Cref{def:nonlinear-approx}), then wavelets satisfy
\[
\|f-f_N\|_{L^2}^2=\mathcal{O}(N^{-1}),
\]
whereas shearlets (and similarly curvelets) achieve the nearly optimal rate
\[
\|f-f_N\|_{L^2}^2=\mathcal{O}\!\left(N^{-2}(\log N)^3\right),
\]
see \cite{Mallat1998,CandesDonoho2004,Guo2007,Kutyniok2011,Kutyniok2012}. The geometric reason behind this improvement can be summarized as follows. At a fixed scale $j$, the number of wavelets whose support substantially intersects the discontinuity curve is of order $2^j$ \cite{Mallat1998}, since wavelets do not possess directional selectivity. In contrast, shearlets combine anisotropic scaling with orientation, so that only those elements whose orientation is adapted to the local tangent of the discontinuity contribute significantly to the approximation. Taking both spatial localization and orientation into account, the number of shearlets whose support substantially intersects the discontinuity curve is therefore reduced to order $2^{j/2}$ \cite{Kutyniok2012}. This reduction in the number of significant coefficients is the mechanism responsible for the better sparsity and approximation properties of shearlet systems.

\subsection{Local coherence for shearlets} \label{subsecLocalCoherence}
We now aim to bound the $k$-th local coherence with Fourier measurements, when shearlets are used as a sparsifying system:
$$\mu_k^{\rm loc}(\{\psi_i\}_{i}, \{\varphi_l\}_{l}) =  \sup_{i} |\langle \psi_i, \varphi_k\rangle|.$$
When the measurement frame consists of complex exponentials $\{\varphi_l\}_{l}$, the inner product $\langle \psi_i, \varphi_l\rangle$ is $\hat{\psi}_i$ evaluated at the corresponding measurement point in the frequency space depending on $l$, where 
$$\hat{f}(\xi) = \int_{\mathbb{R}^2} f(x) e^{-2\pi i \langle\xi, x\rangle} \, dx, \quad \xi \in \mathbb{R}^2,$$ 
denotes the Fourier transform of a complex-valued function $f\in L^1(\mathbb R^2)$.
Furthermore, we look at the case where the sparsifying frame consists of shearlets, as discussed in the previous section. We assume that the elements of the shearlet frame have certain decay properties in frequency (see \Cref{coneAdaptedGenerators}). Such shearlets have their essential supports in frequency concentrated on dyadic coronae, which allows us to control their decay in frequency. As the next result shows, we can bound the Fourier transform of these shearlets evaluated at an integer grid. The proof is a direct calculation based on the properties of the shearlet system and we provide it for completeness.

\begin{proposition}\label{localCoherenceProp}
Let the sparsifying frame $\{\psi_i\}_{i \in I}$ consist of coarse-scale, interior cone-adapted and boundary shearlets, as defined in \Cref{coneAdaptedGenerators} with parameters $\alpha>\gamma\geq\frac{3}{4}$ and $q>q'>0$, $q>r>0$. Set
$$
C
=
\max\left\{
\frac{2^{\frac{3}{2}}}{(q')^{\frac{3}{4}}}, \frac{2^{\frac{3}{4}}}{r^{\frac{3}{4}}}\right\}.$$
Then, for every $i \in I$ and $n \in \mathbb{Z}^2$, the following holds.
\begin{enumerate}
    \item If $\psi_i$ is a coarse-scale shearlet, then
    \[
    |\hat{\psi_i}(n)|
    \le
    \begin{cases}
        1 & \text{if $n=(0,0)$,}\\
        \min\{1,|r n_1|^{-\gamma}\}\,\min\{1,|r n_2|^{-\gamma}\}& \text{if $n\neq(0,0)$.}
    \end{cases}
    \]
    
    \item If $\psi_i$ is an interior or boundary shearlet, then
    \[
    |\hat{\psi_i}(n)|
    \le
    \begin{cases}
        0 & \text{if $n=(0,0)$,}\\
        C\|n\|_\infty^{-\frac{3}{4}}& \text{if $n\neq(0,0)$.}
    \end{cases}
    \]
    \end{enumerate}
\end{proposition}

\begin{proof}
Let $\Phi$, $\psi^{(1)}$, $\psi^{(1,b)}$ and $\psi^{(2,b)}$ denote feasible shearlet generators in the sense of \Cref{coneAdaptedGenerators}.
    \begin{enumerate}
    \item
    Since $\psi_i$ is a coarse-scale shearlet, we have $\psi_i = \widetilde\psi_{-1,k}$ and by definition, $(\widetilde\psi_{-1,k})^{\wedge}(\xi)=\widehat\Phi(\xi)\,e^{-2\pi i\langle \xi,k\rangle}$. If $n=(0,0)$, then $|\widehat\Phi(0,0)|=1$ by assumption. If $n\neq (0,0)$, the feasible coarse-scale
    condition gives
    $$
    |\widehat\Phi(n_1,n_2)|
    \le
    \min\{1,|r n_1|^{-\gamma}\}\,\min\{1,|r n_2|^{-\gamma}\}.
    $$
    
    \item
    Let $n = (0, 0)$. Since $\psi_i$ is not a coarse-scale shearlet, we denote its scale by $j\ge 0$. For $|\ell|<\lceil 2^{j/2} \rceil$ and $k \in \mathbb{Z}^2$, if $\psi_i$ is an interior horizontal shearlet, then $\psi_i = \psi^{(1)}_{j,\ell,k}$. So we have
    $$
    (\psi^{(1)}_{j,\ell,k})^{\wedge}(0,0)
    =
    2^{-\frac{3}{4}j}\,
    \widehat{\psi}^{(1)}\bigl(A_{(1)}^{-j}B_{(1)}^{-\ell}(0,0)\bigr)\,
    e^{-2\pi i\langle A_{(1)}^{-j}B_{(1)}^{-\ell}(0,0),\,k\rangle}
    =
    2^{-\frac{3}{4}j}\,\widehat{\psi}^{(1)}(0,0)=0.
    $$
    Similarly for the cases where $\psi_i$ is either an interior vertical shearlet or a boundary shearlet, we have $(\psi^{(2)}_{j,\ell,k})^{\wedge}(0,0) = (\psi^{(1,b)}_{j,\lceil 2^{j/2} \rceil,k})^{\wedge}(0,0) =(\psi^{(2,b)}_{j,-\lceil 2^{j/2} \rceil,k})^{\wedge}(0,0) = 0$. Thus $\hat{\psi_i}(0, 0) = 0.$
    
    Now fix $n = (n_1, n_2) \in \mathbb{Z}^2\setminus\{(0, 0)\}$ and let $j \geq 0$ denote the scale. For all $k \in \mathbb{Z}^2$ and $|\ell| < \lceil 2^{j/2} \rceil$ we have,
    $$
    \big|(\psi^{(1)}_{j,\ell,k})^{\wedge}(n)\big|
    =
    2^{-\frac{3}{4}j}
    \left|
    \widehat\psi^{(1)}
    \big(nA_{(1)}^{-j}B_{(1)}^{-\ell}\big)
    \right|.
    $$

    The first component of $nA_{(1)}^{-j}B_{(1)}^{-\ell}$ is $2^{-j}n_1$. Hence the feasible generator bound gives
    $$
    \big|(\psi^{(1)}_{j,\ell,k})^{\wedge}(n)\big|
    \leq
    2^{-\frac{3}{4}j}
    \min\left\{
    1,
    \big(q'2^{-j} |n_1|\big)^{-\gamma}
    \right\}.
    $$
    If $|n_1| \geq \frac{\|n\|_\infty}{2}$, then with $t = 2^{j}$ and $a = q'|n_1|$,
    $$
    2^{-\frac{3}{4}j}
    \min\left\{
    1,
    \big(q'2^{-j}|n_1|\big)^{-\gamma}
    \right\}
    =
    t^{-\frac{3}{4}}\min\left\{1, \left(\frac{t}{a}\right)^\gamma\right\}
    \leq
    a^{-\frac{3}{4}},
    $$
    where the last inequality follows by considering separately the two cases $t \geq a$ for which we have $t^{-\frac{3}{4}} \leq a^{-\frac{3}{4}}$ and the case $t < a$ for which we have $t^{\gamma - \frac{3}{4}} a^{-\gamma} \leq a^{-\frac{3}{4}}$ by using $\gamma \geq \frac{3}{4}$. Therefore,
    $$
    \bigl|(\psi^{(1)}_{j,\ell,k})^{\wedge}(n)\bigr|
    \leq
    (q'|n_1|)^{-\frac{3}{4}}
    \leq
    \frac{2^{\frac{3}{4}}}{(q')^{\frac{3}{4}} \|n\|_\infty^{\frac{3}{4}}}.
    $$

    It remains to consider the case $|n_1| < \frac{\|n\|_\infty}{2}$. Then $\|n\|_\infty = |n_2|$. The second component of $nA_{(1)}^{-j}B_{(1)}^{-\ell}$ is,
    $$
    2^{-\frac{j}{2}}n_2 - \ell2^{-j} n_1
    =
    2^{-\frac{j}{2}}(n_2 - 2^{-\frac{j}{2}} \ell n_1).
    $$
    Since $|\ell| < \lceil 2^{\frac{j}{2}} \rceil$, we have $|\ell| \leq \lceil 2^{\frac{j}{2}}\rceil - 1 < 2^{\frac{j}{2}}$ and thus $|2^{-\frac{j}{2}}\ell| < 1$. Therefore by reverse triangle inequality,
    \begin{equation} \label{eq:secCoordIneq}
        |n_2 - 2^{-\frac{j}{2}}\ell n_1|
        \geq 
        |n_2|-|2^{-\frac{j}{2}}\ell|\,|n_1|
        >
        \|n\|_\infty - \frac{\|n\|_\infty}{2}
        =
        \frac{\|n\|_\infty}{2}.
    \end{equation}
    Using \eqref{eq:secCoordIneq} and the decay in the vertical component in \Cref{coneAdaptedGenerators} we get,
    $$
    \big|(\psi^{(1)}_{j,\ell,k})^{\wedge}(n)\big|
    \leq
    2^{-\frac{3}{4}j}
    \min\left\{
    1,
    \left(r 2^{-\frac{j}{2}} \frac{\|n\|_\infty}{2}\right)^{-\gamma}
    \right\}.
    $$
    With $t = 2^{\frac{j}{2}}$ and $a = \frac{r \|n\|_\infty}{2}$, the right-hand side becomes
    \begin{equation} \label{eq:secCoordUpBound}
        t^{-\frac{3}{2}}
        \min\left\{
        1,
        \left(\frac{t}{a}\right)^\gamma
        \right\},
    \end{equation}
    which we again claim can be upper bounded by $a^{-\frac{3}{4}}$. Indeed, if $a < 1$, \eqref{eq:secCoordUpBound} is at most $1$ and we have $1 \leq a^{-\frac{3}{4}}$. If $a \geq 1$ and $t \geq a$, then $t^{-\frac{3}{2}}\leq a^{-\frac{3}{2}}\leq a^{-\frac{3}{4}}$. If $a \geq 1$ and $t < a$ then,
    $$
    t^{-\frac{3}{2}}\left(\frac{t}{a}\right)^\gamma
    =
    t^{\gamma - \frac{3}{2}}a^{-\gamma}
    \leq
    a^{-\frac{3}{4}},
    $$
    where the last inequality follows by considering the two cases $\gamma \leq \frac{3}{2}$ and $\gamma > \frac{3}{2}$ separately. Indeed, for the first case, we have $t^{\gamma - \frac{3}{2}} \leq 1$ and $a^{-\gamma} \leq a^{-\frac{3}{4}}$ since $\gamma \geq \frac{3}{4}$ and for the second case we have $t^{\gamma - \frac{3}{2}} \leq a^{\gamma - \frac{3}{2}}$ and $a^{-\frac{3}{2}} \leq a^{-\frac{3}{4}}$. Hence
    \begin{equation} \label{eq:secCaseBound}
        \bigl|(\psi^{(1)}_{j,\ell,k})^{\wedge}(n)\bigr|
        \leq
        a^{-\frac{3}{4}}
        =
        \left(\frac{r \|n\|_\infty}{2}\right)^{-\frac{3}{4}}
        =
        \frac{2^{\frac{3}{4}}}{r^{\frac{3}{4}} \|n\|_\infty^{\frac{3}{4}}}.    
    \end{equation}

    The boundary shearlets can be treated similarly with the only change being $\ell = \pm \lceil 2^{j/2}\rceil$. For $\psi^{(1,b)}_{j,\lceil 2^{j/2}\rceil,k}$, the decay on the first component gives the first case. If $|n_1|\geq \frac{\|n\|_\infty}{4}$, this gives
    $$
    \bigl|(\psi^{(1,b)}_{j,\lceil 2^{j/2}\rceil,k})^{\wedge}(n)\bigr|
    \leq
    (q'|n_1|)^{-\frac{3}{4}}
    \leq
    \frac{2^{\frac{3}{2}}}{(q')^{\frac{3}{4}}\|n\|_\infty^{\frac{3}{4}}}.
    $$
    If $|n_1|<\frac{\|n\|_\infty}{4}$, then $\|n\|_\infty=|n_2|$ and the second component is of the form
    $$
    2^{-\frac{j}{2}}\left(n_2-\lceil 2^{j/2}\rceil 2^{-\frac{j}{2}}n_1\right).
    $$
    Since $2^{-j/2}\lceil 2^{j/2}\rceil\leq 2$, by reverse triangle inequality we have,
    $$
    \left|n_2-\lceil 2^{j/2}\rceil 2^{-\frac{j}{2}}n_1\right|
    \geq
    |n_2|-2|n_1|
    >
    \|n\|_\infty - 2 \cdot \frac{\|n\|_\infty}{4}
    =
    \frac{\|n\|_\infty}{2}.
    $$
    From this point following the argument that led to \eqref{eq:secCaseBound} we obtain the same bound. The proof for interior vertical shearlets and the vertical boundary shearlet $\psi^{(2,b)}_{j,-\lceil 2^{j/2}\rceil,k}$ is essentially the same with the roles of $n_1$ and $n_2$ interchanged. Therefore, for every interior or boundary shearlet and every $n\neq(0,0)$,
    $$
    |\hat{\psi_i}(n)|
    \leq 
    \max\left\{
    \frac{2^{\frac{3}{4}}}{(q')^{\frac{3}{4}}},
    \frac{2^{\frac{3}{2}}}{(q')^{\frac{3}{4}}},
    \frac{2^{\frac{3}{4}}}{r^{\frac{3}{4}}}
    \right\} 
    =
    \max\left\{
    \frac{2^{\frac{3}{2}}}{(q')^{\frac{3}{4}}},
    \frac{2^{\frac{3}{4}}}{r^{\frac{3}{4}}}
    \right\} 
    =
    C\|n\|_\infty^{-\frac{3}{4}}.
    $$
    This proves the second claim and completes the proof.
    \end{enumerate}
\end{proof}

In the setting of \Cref{localCoherenceProp}, we now pass to the unit-square setting $\cH = L^2([0,1]^2)$, namely, we consider square-integrable functions on $[0,1]^2$. We let $\{\psi_i\}_i$ be the shearlet frame as in the assumptions of \Cref{localCoherenceProp} (considering only the shearlet elements whose support is essentially contained in $[0,1]^2$). We let $\{\varphi_l\}_{l\in\N}$ be the orthonormal basis of $\cH$ consisting of Fourier exponentials, namely, $\varphi_{l}(x) = e^{2\pi i  x\cdot n_l}$, where $l\in\N\mapsto n_l\in\Z^2$ is a bijection. In this setting, \Cref{localCoherenceProp} gives\footnote{In the case of compactly supported shearlets (\Cref{thm:compact-shearlets}), this derivation is fully rigorous, because $\langle\psi_i,\varphi_l\rangle=\hat\psi_i(n_l)$. In the case of bandlimited shearlets (\Cref{thm:bandLimitedFrame}), here we discard the fact that $\operatorname{supp}\psi_i\not\subseteq [0,1]^2$. However, since $\operatorname{supp}\psi_i$ is very well localized, this yields exponentially small errors, which do not affect the overall bound.}
\begin{equation}
\label{eq:shear_loc_cohe_bound}
    \mu_k^{\rm loc}(\{\psi_i\}_i, \{\varphi_{l}\}_{l \in \N})
    \lesssim (1+\|n_k\|_\infty)^{-3/4}, \quad k\in\N.
\end{equation}
In view of \eqref{eq:muk-omegak} and \eqref{eq:nul-omegal}, this bound will guide the choice of the probability density for sampling the Fourier measurements used in the next section.

\begin{remark} \label{waveletCoherenceRemark}
    For a sparsifying frame $\{\psi_i\}_{i}$ of two-dimensional separable Daubechies wavelets, it has been shown that $\mu_k^{\rm loc}(\{\psi_i\}, \{\varphi_{l}\})$ has a decay in $k$ of order $\mathcal{O}((1+\|n_k\|_\infty)^{-1})$ \cite[Theorem 2.3]{Jones2016}.
\end{remark}

\begin{remark}
\label{rmk:optimality_loc_coherence}
We remark that the estimate \eqref{eq:shear_loc_cohe_bound} is essentially optimal. Indeed, from \Cref{coneAdaptedGenerators}, we see that horizontal cone-adapted shearlets at scale $j$ are essentially supported in the two strips $\{c2^{j}\leq |\xi_1|\leq C2^{j}\}$ for some constants $c,C>0$. Given a Fourier coefficient $n_k=(n_k^1,n_k^2)$ in such a strip, we can consider a suitable shearlet element $\psi_{i_0}$, to be chosen later. We then have that
\begin{align*}
    \mu_k^{\rm loc}(\{\psi_i\}_{i },\{\varphi_l\}_{l }) \geq |\langle \psi_{i_0},\varphi_k \rangle| = |\widehat{\psi_{i_0}}(n_k)|.
\end{align*}
Littlewood-Paley conditions on shearlet generators (see, for example, (1.6) and (1.8) in \cite{Guo2007}), required to enforce the frame property, imply that there exist $c'>0$ and $i_0=i_0(n_k)$, indexing a horizontal cone-adapted shearlet, such that
\begin{align*}
    |\widehat{\psi_{i_0}}(n_k)| \geq c'2^{-3/4 j},
\end{align*}
where the factor $2^{-3/4j}$ is due to the scaling of shearlets with respect to the parameter $j$ (see \Cref{coneAdaptedGenerators}). Finally, we have that
\begin{align*}
2^{-3/4 j} \gtrsim |n_k^1|^{-3/4} \geq (1+\|n_k\|_{\infty})^{-3/4}.
\end{align*}
This implies the claimed lower bound on the local coherence for frequencies $n_k$ in the horizontal cones. The same argument applies \textit{verbatim} to vertical cone-adapted shearlets.
\end{remark}

\subsection{Localization factor for cone-adapted shearlets} \label{subsecLocalizationFactor}
We derive here some abstract results bounding the localization factor related to two arbitrary frames in a separable Hilbert space $\mathcal{H}$.
\begin{lemma}\label{localizationLemma}
    Let $\{\psi_i\}_{i\in I}$ be a Parseval sparsifying frame such that $\|\psi_i\|_{\mathcal{H}} \leq C$ for some $C > 0$ and all $i \in I$. Then, the localization factor (see \Cref{def:localization_factor}) can be bounded by
    $$\eta_{s, D} \leq C\sqrt{s} + C \sup_{\substack{\Delta \subseteq I, \\ |\Delta| = s}} \frac{1}{\sqrt{A_\Delta}} \sum_{i \in I \setminus \Delta}\sqrt{|\langle \psi_{i}, \psi_{l_{i, \Delta}}\rangle|},$$
    where $A_\Delta$ is the optimal lower frame bound of the finite subframe $\{\psi_i\}_{i\in \Delta}$ and 
    $$l_{i, \Delta} = \argmax_{k \in \Delta} |\langle \psi_i, \psi_k\rangle|.$$
\end{lemma}

\begin{proof}

    Let $\Delta \subseteq I$ be such that $|\Delta| = s$ and denote $\mathcal{W}_\Delta = \operatorname{Im}(D^*P_\Delta)$. Take $g \in \mathcal{\mathcal{W}}_\Delta$ such that $\|g\|_{\mathcal{H}} = 1$. There exists some $x \in \ell^2(I)$ such that $g = D^*P_\Delta x$. Furthermore, since $g$ has unit norm, we estimate $\|Dg\|_1$ in terms of $\|g\|_{\mathcal{H}}$, which equals one. First, by Cauchy-Schwarz inequality we get:
    \begin{align}\label{csInequality}
        \|Dg\|_1 &= \sum_{i \in I} |\langle g, \psi_i\rangle| = \sum_{i \in I} |\langle P_{\mathcal{W}_\Delta} g, \psi_i\rangle| = \sum_{i \in I}| \langle g, P_{\mathcal{W}_\Delta}\psi_i\rangle| \nonumber \\
        & \leq \sum_{i \in I} \|g\|_{\cH} \cdot \|P_{\mathcal{W}_\Delta} \psi_i\|_\cH = \sum_{i \in I} \|P_{\mathcal{W}_\Delta} \psi_i\|_\cH \\
        &= \sum_{i \in \Delta} \|\psi_i\|_\cH + \sum_{i \in I \setminus \Delta} \|P_{\mathcal{W}_\Delta} \psi_i\|_\cH \leq Cs + \sum_{i \in I \setminus \Delta} \|P_{\mathcal{W}_\Delta} \psi_i\|_\cH.\nonumber
    \end{align}

    The interesting case occurs when $i \in I \setminus \Delta$. We can write the projection of the corresponding frame element to the subspace $\mathcal{W}_\Delta$ using the usual frame reconstruction formula for the subframe $\{\psi_i\}_{i\in \Delta}$. Denote the frame operator of the subframe $\{\psi_i\}_{i \in \Delta}$ as $S_{\mathcal{W}_\Delta}$ so that $\{S_{\mathcal{W}_\Delta}^{-1}\psi_i\}_{i \in \Delta}$ is the canonical dual frame of the subframe. For a fixed index $i \in I \setminus \Delta$ we can write the projection of its corresponding frame element to $\mathcal{W}_\Delta$ as:
    $$P_{\mathcal{W}_\Delta} \psi_{i} = \sum_{j \in \Delta} \langle P_{\mathcal{W}_\Delta}\psi_{i}, \psi_j\rangle S_{\mathcal{W}_\Delta}^{-1} \psi_j =\sum_{j \in \Delta} \langle \psi_{i}, \psi_j\rangle S_{\mathcal{W}_\Delta}^{-1} \psi_j.$$
    We now compute the projection's norm as:
    \begin{align*}
     \|P_{\mathcal{W}_\Delta} \psi_{i}\|_\cH^2 &= \langle P_{\mathcal{W}_\Delta} \psi_{i}, P_{\mathcal{W}_\Delta} \psi_{i}\rangle = \langle P_{\mathcal{W}_\Delta} \psi_{i},  \psi_{i}\rangle = \bigg\langle \sum_{j \in \Delta} \langle \psi_{i}, \psi_j\rangle S_{\mathcal{W}_\Delta}^{-1} \psi_j,  \psi_{i}\bigg\rangle \\
        &= \bigg|\sum_{j \in \Delta} \langle \psi_{i}, \psi_j\rangle \cdot \langle S_{\mathcal{W}_\Delta}^{-1} \psi_j,  \psi_{i}\rangle\bigg| \leq \sum_{j \in \Delta} |\langle \psi_{i}, \psi_j\rangle| \, |\langle S_{\mathcal{W}_\Delta}^{-1} \psi_j,  \psi_{i}\rangle| \\
        &\leq \sum_{j \in \Delta} |\langle \psi_{i}, \psi_j\rangle| \cdot \|S_{\mathcal{W}_\Delta}^{-1} \psi_j\|_\cH \cdot \|\psi_{i}\|_\cH \leq \sum_{j \in \Delta} |\langle \psi_{i}, \psi_j\rangle| \cdot \|S_{\mathcal{W}_\Delta}^{-1}\| \cdot \| \psi_j\|_\cH \cdot \|\psi_{i}\|_\cH \\
        &\leq C^2\|S_{\mathcal{W}_\Delta}^{-1}\| \sum_{j \in \Delta} |\langle \psi_{i}, \psi_j\rangle|.
    \end{align*}
 When we sum over the projections we get:
    $$\|Dg\|_1   \leq Cs + C\sqrt{\|S_{\mathcal{W}_\Delta}^{-1}\|}\sum_{i \in I \setminus \Delta}\sqrt{\sum_{j \in \Delta} |\langle \psi_{i}, \psi_j\rangle|}.$$

    By \Cref{frameBound}, the smallest eigenvalue of $S_{\mathcal{W}_\Delta}$ is $A_\Delta$. Furthermore, since both $S_{\mathcal{W}_\Delta}$ and $S_{\mathcal{W}_\Delta}^{-1}$ are self-adjoint operators,  their operator norms are the absolute values of their respective largest eigenvalues. Since the reciprocals of the eigenvalues of $S_{\mathcal{W}_\Delta}$ are the eigenvalues of $S_{\mathcal{W}_\Delta}^{-1}$, we have $\|S_{\mathcal{W}_\Delta}^{-1}\| = \frac{1}{A_\Delta}$. Substituting this and using the fact that for each of the $s$ indices $j$ such that $j \in \Delta$ we have by definition $|\langle \psi_{i}, \psi_j \rangle| \leq |\langle \psi_{i}, \psi_{l_{i, \Delta}}\rangle|$, we get
    $$\|Dg\|_1 \leq Cs + \frac{C}{\sqrt{A_\Delta}}\displaystyle\sum_{i \in I\setminus\Delta}\sqrt{\sum_{j \in \Delta} |\langle \psi_{i}, \psi_j\rangle|} \leq Cs + \frac{C}{\sqrt{A_\Delta}}\displaystyle\sum_{i \in I\setminus\Delta}\sqrt{s \cdot |\langle \psi_{i}, \psi_{l_{i, \Delta}}\rangle|}.$$
    Thus, recalling the expression of the localization factor given in \Cref{def:localization_factor}, we have
    $$\eta_{s,D} \leq  C\sqrt{s} + C \sup_{\substack{\Delta \subseteq I, \\ |\Delta| = s}} \frac{1}{\sqrt{A_\Delta}} \sum_{i \in I \setminus \Delta}\sqrt{|\langle \psi_{i}, \psi_{l_{i, \Delta}}\rangle|}.$$
    This concludes the proof.
\end{proof}

\begin{remark}
    Notice that the bound given in \Cref{localizationLemma} is suboptimal.  The suboptimality mainly comes from the application of the Cauchy-Schwarz inequality at \eqref{csInequality}. This can be seen by the fact that in the best scenario where the sparsifying frame $\{\psi_i\}_{i \in I}$ is an orthonormal basis, the localization factor should be $\eta_{s, D} = 1$, while after \eqref{csInequality} for an $s$-sparse $\Delta \subseteq I$ we already obtain a term of the form
    $$\frac{\displaystyle\sum_{i \in \Delta} \|\psi_i\|_2}{\sqrt{s}} = \frac{|\Delta|}{\sqrt{s}} = \sqrt{s}$$ 
    in the upper bound.
\end{remark}

To apply \Cref{localizationLemma}, we need to estimate for $\Delta \subseteq I$ such that $|\Delta| = s$, an upper bound for $\displaystyle\sum_{i\in I \setminus \Delta} \sqrt{|\langle \psi_i, \psi_{l_{i, \Delta}}\rangle|}$ and a lower bound for $A_\Delta$. The first can be done using a concept called \emph{intrinsic localization} \cite{Fornasier2005}, as we now discuss. Next, we will comment on $A_\Delta$.

\begin{definition}
        A frame $\{\psi_i\}_{i \in I}$ for the Hilbert space $\mathcal{H}$ is \emph{intrinsically localized} with parameters $c > 0$ and $L > 1$ if
        $$|\langle \psi_i, \psi_l \rangle| \leq \frac{c}{(1 + |i-l|)^L},\quad \forall i,\, l \in I.$$
    \end{definition}

    \begin{remark}
        Wavelets can be shown to satisfy the above definition of intrinsic localization \cite{Cordero2004} where the parameter $L$ depends on the regularity of the generating wavelets. For what are called \emph{parabolic molecules}, like shearlets or curvelets, an analogous property is known \cite[Theorem 2.9]{Grohs2013} (see Appendix~\ref{parabolicMoleculesAppendix}).
    \end{remark}

    \begin{lemma} \label{intrinsicLocalizationLemma}
        Let $\{\psi_i\}_{i\in I}$ be a frame for the Hilbert space $\mathcal{H}$ that satisfies the intrinsic localization property with parameters $c > 0$ and $L > 2$. Then 
        \begin{equation}\label{eq:intrinsic-loc}
            \sup_{\substack{ \Delta \subseteq I,\\|\Delta| = s}}\sum_{i \in I \setminus \Delta}\sqrt{ |\langle \psi_i, \psi_{l_{i, \Delta}}\rangle|} \leq \frac{4s\sqrt{c}}{L-2},
        \end{equation}
        where $l_{i, \Delta}$ is defined as in \Cref{localizationLemma}.
    \end{lemma}

    \begin{proof}
        Let $\Delta \subseteq I$ be such that $|\Delta| = s$. We want to bound the sum:
        $$\sum_{i \in I \setminus \Delta}\sqrt{ |\langle \psi_i, \psi_{l_{i, \Delta}}\rangle|}.$$
        Notice that for all $i \not \in \Delta$ since $l_{i, \Delta} \in \Delta$, we have $|i - l_{i, \Delta}| \geq d(i, \Delta) := \displaystyle \min_{l \in \Delta} |i - l|$. Now using the intrinsic localization of the frame we get:
        \begin{align*}
            \sum_{i \in I \setminus \Delta}\sqrt{ |\langle \psi_i, \psi_{l_{i, \Delta}}\rangle|} &\leq \sum_{i \in I \setminus \Delta}\sqrt{ \frac{c}{(1 + |i - l_{i, \Delta}|)^L}} = \sum_{i \in I \setminus \Delta} \frac{\sqrt{c}}{(1 + |i - l_{i, \Delta}|)^{L/2}} \\
            &\leq \sum_{i \in I \setminus \Delta} \frac{\sqrt{c}}{(1 + d(i,  \Delta))^{L/2}} = \sum_{i \in I \setminus \Delta} \sum_{\substack{k \geq 1, \\ k = d(i, \Delta)}} \frac{\sqrt{c}}{(1 + k)^{L/2}} \\
            &\leq \sum_{k = 1}^\infty \sum_{i \in I \setminus \Delta}\sum_{\substack{l \in \Delta, \\ |i - l| = k}} \frac{\sqrt{c}}{(1 + k)^{L/2}} = \sum_{k = 1}^\infty \sum_{l \in \Delta}\sum_{\substack{i \in I \setminus \Delta, \\ |i - l| = k}} \frac{\sqrt{c}}{(1 + k)^{L/2}} \\
            &\leq \sum_{k = 1}^\infty \sum_{l \in \Delta} \frac{2\sqrt{c}}{(1 + k)^{L/2}} = \sum_{k = 1}^\infty \frac{2s\sqrt{c}}{(1 + k)^{L/2}} \leq \frac{4s\sqrt{c}}{L-2},
        \end{align*}
        where in the last inequality we used the bound $\sum_{n=2}^\infty \frac{1}{n^a}\leq \frac{1}{a-1}$ for $a>1$.

        Taking the supremum over all subsets $\Delta \subseteq I$ with $|\Delta|=s$, we obtain the inequality.
    \end{proof}
    The bound obtained in \eqref{eq:intrinsic-loc} is not optimal, because of the factor $s$ instead of $\sqrt{s}$. For the case where the sparsifying frame consists of parabolic molecules (see Appendix~\ref{parabolicMoleculesAppendix}), we can obtain a result analogous to \Cref{intrinsicLocalizationLemma}.

    \begin{lemma}\label{lem:intrinsic-mol}
        Let $(I, \Phi_I)$ be a parametrization with $\Phi_I$ injective and $(\Lambda^0, \Phi^0)$ be the canonical parametrization as in \Cref{parametrizationDef} such that $\Phi_I(I) \subseteq \Phi^0(\Lambda^0)$. Also let $L > 4$ and $\{\psi_i\}_{i \in I}$ be a frame of parabolic molecules of order $(R, M, N_1, N_2)$ as in \Cref{parabolicMoleculeDef} such that
        $$R \geq 2L, \quad M > 4L - \frac{5}{4}, \quad N_1 \geq 2L + \frac{3}{4}, \quad N_2 \geq 2L.$$
         Then, for the index distance $\omega$ as defined in \Cref{indexDistDef} we have,
         $$C_L := \sup_{l \in I} \sum_{i \in I}\omega(\Phi_{I}(i), \Phi_I(l))^{-L/2} < \infty,$$
         and for some constant $C' > 0$,
         $$\sup_{\substack{ \Delta \subseteq I,\\|\Delta| = s}}\sum_{i \in I \setminus \Delta}\sqrt{ |\langle \psi_i, \psi_{l_{i, \Delta}}\rangle|} \leq C'C_L s.$$
    \end{lemma}

    \begin{proof}
    Since the mapping $\Phi_I$ associated to the parametrization is injective and $\Phi_I(I) \subseteq \Phi^0(\Lambda^0)$, we can write
        $$C_L \leq \sup_{l \in I} \sum_{\gamma \in \Lambda^0}\omega(\Phi^0(\gamma), \Phi_I(l))^{-L/2} \leq \sup_{\lambda \in \Lambda^0} \sum_{\gamma \in \Lambda^0}\omega(\Phi^0(\gamma), \Phi^0(\lambda))^{-L/2},$$
    where the right-hand side is bounded thanks to the $\frac{L}{2}$-admissibility of the canonical parametrization due to \Cref{LadmissibleLemma}, since $\frac{L}{2} > 2$. Now let $\Delta \subseteq I$ be such that $|\Delta| = s$. By the assumption on the order of our parabolic molecules $\{\psi_i\}_{i \in I}$, we can apply \Cref{moleculeIntrinsicLocTheorem} and obtain
    \begin{align*}
        \sum_{i \in I \setminus \Delta}\sqrt{ |\langle \psi_i, \psi_{l_{i, \Delta}}\rangle|} &\leq \sum_{i \in I \setminus \Delta}\sqrt{C\omega(\Phi_I(i), \Phi_I(l_{i, \Delta}))^{-L}} = \sqrt{C} \sum_{i \in I \setminus \Delta}\omega(\Phi_I(i), \Phi_I(l_{i, \Delta}))^{-L/2},
    \end{align*}
   for some constant $C > 0$. Now the proof follows similarly to that of \Cref{intrinsicLocalizationLemma}. Since $l_{i, \Delta} \in \Delta$, we have
    \begin{align*}
        \sum_{i \in I \setminus \Delta}\omega(\Phi_I(i), \Phi_I(l_{i, \Delta}))^{-L/2} &= \sum_{l \in \Delta}\sum_{\substack{i \in I \setminus \Delta, \\ l_{i, \Delta} = l}} \omega(\Phi_I(i), \Phi_I(l))^{-L/2} \leq \sum_{l \in \Delta}\sum_{i \in I} \omega(\Phi_I(i), \Phi_I(l))^{-L/2} \\
        &\leq \sum_{l \in \Delta} \sup_{k \in \Delta}\sum_{i \in I} \omega(\Phi_I(i), \Phi_I(k))^{-L/2} = s C_L.
    \end{align*}
    Letting $C' = \sqrt{C}$ and taking supremum over the $s$-sparse $\Delta \subseteq I$ we obtain,
    $$\sup_{\substack{ \Delta \subseteq I,\\|\Delta| = s}}\sum_{i \in I \setminus \Delta}\sqrt{ |\langle \psi_i, \psi_{l_{i, \Delta}}\rangle|} \leq C'C_L s.$$
    This concludes the proof.
    \end{proof}
    Let us now continue with the discussion about the factor $\frac{1}{A_\Delta}$ in \Cref{localizationLemma}. 

    \begin{remark}\label{rem:A}
    \Cref{waveletLowerBound} tells us that, for one-dimensional bandlimited wavelets, $A_\Delta$ can be bounded from below by an expression of the form $C^{j_0}$ for some $1 \geq C > 0$ where $j_0$ is the finest scale present in the index set $\Delta \subseteq I$.  Analogous arguments may be possible for 2D wavelets and directional systems, but the available 1D estimates already indicate the difficulty. However, the main issue is the possible dependence of the constant $C$ on the largest number of translation parameters for a scale that is present in $\Delta$, because the constant $A$ in \Cref{waveletLowerBound} can potentially be exponentially small even in the number of translation parameters, as shown in \cite[Proposition 2.3]{Christensen2001Exponential}. Thus, an upper bound for $\frac{1}{A_\Delta}$ obtained by a generalization of \Cref{waveletLowerBound} to 2D wavelets/directional wavelets would potentially be exponentially large both in the finest scale $j_0$ present in $\Delta$ and in the maximum number of translation parameters corresponding to $j_0$. This would be considerably worse than a term in $|\Delta|$ that is constant, linear or at most polynomial.
\end{remark}

To conclude, we emphasize that the derivations of this section show why finding a useful upper bound for the localization factor $\eta_{s, D}$ for shearlets is complicated. Indeed, recalling that $\eta_{s, D}^2$ appears in the sample complexity rate \eqref{eq:cor:m}, one would ideally aim for a bound with a  constant independent of $s$. However, the bound obtained in \Cref{localizationLemma} gives 
\begin{itemize}
    \item the factor $\sqrt{s}$;
    \item the sum of the square roots of the scalar products between the elements of the frame, which is $\mathcal{O}(s)$ in view of \Cref{intrinsicLocalizationLemma,lem:intrinsic-mol};
    \item the factor $\frac{1}{\sqrt{A_\Delta}}$, for which the known upper bounds can be very large as explained in \Cref{rem:A}.
\end{itemize}

\subsection{Comparing the sample complexities with wavelets and shearlets}\label{sub:comparing}
In this section, we discuss sample complexity estimates for the stable recovery of a cartoon-like image from $m$ randomly sampled Fourier measurements. In particular, we will compare the cases where the sparsifying dictionary is either an orthonormal basis of wavelets or a  frame of shearlets.

We  consider the following finite-dimensional approximation $f_{j_0}$ of a cartoon-like image: given a cartoon-like image $f$, we consider the image $f_{j_0}$ defined as
\begin{align*}
    f_{j_0} = \sum_{0 \leq j\leq j_0} \sum_n \langle f,\psi_{j,n} \rangle\, \psi_{j,n},
\end{align*}
where $(\psi_{j,n})_{j,n}$ denotes for convenience either a wavelet or a shearlet dictionary, $j$ denotes the scale index in both cases, and $n$ denotes the remaining indices. As outlined in \Cref{corollary2.9}, the sample complexity required to recover a compressible signal $f_{j_0}$ scales roughly as follows:
\begin{equation}
\label{eq:sample_complexity_discussion}
    m \gtrsim \eta_{s, D}^2 \|\omega\|_2^2\, s \cdot\text{(log factors)},
\end{equation}
where $\eta_{s, D}$ is the localization factor, $\|\omega\|_2$ is related to the local coherence via the bound
\begin{align*}
    \mu_k^{\rm loc}(\{\psi_{j,n}\}_{j,n},\{\varphi_l\}) \leq \omega_{k},
\end{align*}
and $s$ is the sparsity level of $f_{j_0}$.

This result holds under the assumption that the sensing dictionary $\{\varphi_l\}_l$ is an orthonormal basis of the space of signals under consideration, which may not be the case in the case of Fourier measurements and finite-dimensional subspaces of wavelets and shearlets. However, as also mentioned in Appendix~\ref{metaResultAppendix}, an additional ingredient in this case is given by the so-called balancing property (see \Cref{balancingProperty}), a concept that is also connected to the stable sampling rate. In \cite[Theorem~4.1]{Adcock2014}, estimates on the balancing property for wavelets are used to show that, if the Fourier measurements bandwidth is chosen as $c 2^{j_0}$ for a suitable $c>1$, then the same sample complexity \eqref{eq:sample_complexity_discussion} holds true also in this case. Analogous estimates are proved in \cite[Theorem~4.1]{Ma2017} for shearlets. As discussed in the paper, the obtained stable sampling rate might be slightly worse for shearlets. In what follows, we consider the most optimistic scenario where the sampling rate for shearlets scales analogously to the one for wavelets, namely we consider in both cases a bandwidth of $c2^{j_0}$ for Fourier measurements.

Recall that the localization factor satisfies $\eta_{s,D}=1$ for an orthonormal basis. As also mentioned at the end of \Cref{sec:compressed_sensing_setting}, in general, for highly redundant frames, $\eta_{s,D}\geq 1$ can be arbitrarily large. Moreover, as discussed in \Cref{subsecLocalizationFactor}, it may be non-trivial to give an upper bound for this quantity in the case of cone-adapted shearlets. For the remainder of this section, we will again ignore these issues, assuming the most optimistic scenario for which $\eta_{s,D}\approx 1$ also for the shearlet frame under consideration.

Concerning the local coherence $\mu^{\rm{loc}}$, we have proved in \Cref{localCoherenceProp} (see also \eqref{eq:shear_loc_cohe_bound}) that the local coherence $\mu_k^{\rm loc,\,shear}$ for shearlets and Fourier measurements scales like
\begin{align*}
    \mu_k^{\rm loc,\,shear} \lesssim (1+\|n_k\|_{\infty})^{-3/4}= \omega_k^{\rm shear},
\end{align*}
and that this bound is essentially optimal, as outlined in \Cref{rmk:optimality_loc_coherence}. Analogously, in \Cref{waveletCoherenceRemark}, we have recalled that the local coherence $\mu_k^{\rm loc,\,wave}$ for wavelets and Fourier measurements scales like
\begin{align*}
    \mu_k^{\rm loc,\,wave} \lesssim (1+\|n_k\|_{\infty})^{-1} = \omega_k^{\rm wave}.
\end{align*}
This implies that
\begin{align*}
    \|\omega^{\rm wave}\|_{2}^2 &=
    \sum_{\|n_k\|_{\infty}\leq c2^{j_0}} (\omega_k^{\rm wave})^2 =
    \sum_{\|n_k\|_{\infty}\leq c2^{j_0}} (1 + \|n_k\|_\infty)^{-2} \\ 
    &\leq
    \sum_{0\leq j\leq j_0} \sum_{\substack{n \in \mathbb{Z}^2, \\ c 2^{j-1}\leq \|n\|_{\infty} \leq c2^{j}}} (1 + \|n\|_\infty)^{-2} \lesssim
    \sum_{0\leq j\leq j_0} \sum_{\substack{n \in \mathbb{Z}^2, \\ c 2^{j-1}\leq \|n\|_{\infty} \leq c2^{j}}} 2^{-2j} \\
     &\approx
    \sum_{0\leq j\leq j_0} 2^{2j} 2^{-2j} = j_0 + 1,
\end{align*}
where we have used that
\begin{align*}
    \#\{n\in\mathbb{Z}^2\colon\ c 2^{j-1}\leq \|n\|_{\infty} \leq c2^{j}\} \approx 2^{2j}.
\end{align*}
On the other hand, we have that
\begin{align*}
    \|\omega^{\rm shear}\|_{2}^2 &=
    \sum_{\|n_k\|_{\infty}\leq c2^{j_0}} (\omega_k^{\rm shear})^2 =\sum_{\|n_k\|_{\infty}\leq c2^{j_0}} (1 + \|n_k\|_\infty)^{-3/2} \\ 
    &\leq
    \sum_{0\leq j\leq j_0} \sum_{\substack{n \in \mathbb{Z}^2, \\ c 2^{j-1}\leq \|n\|_{\infty} \leq c2^{j}}} (1 + \|n\|_\infty)^{-3/2} \lesssim
    \sum_{0\leq j\leq j_0} \sum_{\substack{n \in \mathbb{Z}^2, \\ c 2^{j-1}\leq \|n\|_{\infty} \leq c2^{j}}} 2^{-3j/2} \\
    &\approx
    \sum_{0\leq j\leq j_0} 2^{2j} 2^{-3/2j} \approx 2^{j_0/2}.
\end{align*}

Concerning the sparsity level of cartoon-like images, as remarked at the end of \Cref{sec:introduction_shearlets}, at each scale $j$, the number of non-negligible wavelet coefficients is of order $2^j$, while the number of non-negligible shearlet coefficients is of order $2^{j/2}$. This implies that the sparsity level of $f_{j_0}$ for wavelets and shearlets is roughly
\begin{align*}
    s^{\rm wave} \approx 2^{j_0},\quad
    s^{\rm shear} \approx 2^{j_0/2}.
\end{align*}
This allows us to estimate the sample complexity from \eqref{eq:sample_complexity_discussion} for the recovery of $f_{j_0}$ using wavelets and shearlets. We obtain the following scaling in $j_0$:
\begin{align*}
    m^{\rm wave} &\gtrsim \|\omega^{\rm wave}\|_{2}^2\,s^{\rm wave}
    \cdot\text{(log factors)} \approx j_0 2^{j_0}\cdot\text{(log factors)},\\
    m^{\rm shear} &\gtrsim \|\omega^{\rm shear}\|_{2}^2\,s^{\rm shear}
    \cdot\text{(log factors)} \approx 2^{j_0}\cdot\text{(log factors)},
\end{align*}
Thus, even in the optimistic scenario where the balancing property and localization factor do not affect the bound, the shearlet estimate improves on the wavelet estimate only by a factor of order $j_0$, i.e. logarithmically in the resolution. This is in agreement with our numerical results, shown in \Cref{sectionNumerics}.

\section{Numerical results} \label{sectionNumerics}

\subsection{Introduction to our numerical setup} \label{numericsIntroduction}

The following numerical experiments are not intended as a direct implementation of \Cref{corollary2.9}; rather, they test whether the sparsity advantage of shearlets is reflected in finite-dimensional Fourier recovery experiments.

\subsubsection{Digital frames} \label{numericsDigitalFrames}

We first specify the digital sparsifying and measurement systems used in the simulations. For the sparsifying frame $D$, we use the shearlet transform of the \textsc{ShearLab} toolbox and the two-dimensional stationary wavelet transform, as detailed below. All computations were carried out in \textsc{MATLAB}, version 2024b \cite{Matlab2024b}, on an AMD Ryzen Threadripper PRO 5955WX.

\begin{enumerate}
    \item \textbf{ShearLab shearlet transform:} For the shearlet case, $D$ is the two-dimensional discrete shearlet transform implemented in \textsc{ShearLab} \cite{Kutyniok2016, ShearlabSite}. The underlying shearlet system $\{\psi_\lambda\}_\lambda$ is generated by parabolic scalings and shearings of a small number of prototype functions and forms a frame for $\mathbb C^{K\times K}$. In the discrete implementation, given a digital image $g$, one first computes its discrete Fourier transform $\widehat g$, multiplies $\widehat g$ by the shearlet filters $\widehat\psi_\lambda$ associated with each cone, scale and shearing parameter triple and then applies an inverse Fourier transform to obtain the coefficients $\{\langle g,\psi_\lambda\rangle\}_\lambda$. The frame operator
    $$
    S g = \sum_{\lambda} \langle g,\psi_\lambda\rangle \psi_\lambda
    $$
    is diagonalized by the discrete Fourier transform. In frequency,
    $$
    \widehat{Sg}(\omega) = \sigma(\omega)\,\widehat g(\omega), 
    \quad 
    \sigma(\omega)= \sum_{\lambda} |\widehat\psi_\lambda(\omega)|^2.
    $$
    So $S$ is a Fourier multiplier with symbol $\sigma$ \cite[Section 3.4]{Kutyniok2016}. Hence the canonical dual frame $\{\widetilde\psi_\lambda\}_\lambda$ is obtained simply by
    $$
    \widehat{\widetilde\psi}_\lambda(\omega)
    = \frac{\widehat\psi_\lambda(\omega)}{\sigma(\omega)},
    $$
    that is, by pointwise dividing each shearlet filter by the frame weights $\sigma(\omega)$ in the frequency domain. For each feasible scale and shearing parameter pair, \textsc{ShearLab} returns an array of coefficients that has the same dimensions as the initial digital image, thanks to an appropriate choice of $c$, see \Cref{rem:c1c2}. Thus, \textsc{ShearLab} and other known state-of-the-art digital shearlet transform implementations give highly redundant frames. To compare with the \textsc{ShearLab} case, this motivated us to choose a redundant, translation invariant digital wavelet implementation which we introduce next. 
    \item \textbf{Stationary wavelet transform:}
    For the wavelet case, $D$ is chosen as the two-dimensional stationary wavelet transform available in the \textsc{MATLAB Wavelet Toolbox} \cite{WaveletToolbox}. In contrast to the critically sampled discrete wavelet transform, the stationary wavelet transform does not perform any downsampling and is therefore translation invariant and highly redundant. Given an image $g \in \mathbb C^{K \times K}$, the \textsc{MATLAB} function \texttt{swt2} returns, at each prescribed scale, one approximation subband and three detail subbands (horizontal, vertical, diagonal), all of the same spatial size as the original image. The function $D$ maps $g$ to the collection of all these subbands stacked into a single coefficient vector. Finally, $D^{-1}$ is realized by the corresponding inverse stationary wavelet transform (given by \texttt{iswt2}), which recombines all subbands to yield a reconstruction of $g$. 
\end{enumerate}

Furthermore, to make the total numbers of coefficients comparable, we select a digital shearlet frame with scales from $0$ to $2$, which results in $17$ cone-scale-shearing parameter triples, while for stationary wavelets we select $4$ levels, which results in $4 \cdot 4 = 16$ level-subband pairs. Thus, for each $256 \times 256$ digital image, we obtain $256^2 \cdot 17 = 1,114,112$ shearlet coefficients and $256^2 \cdot 16 = 1,048,576$ wavelet coefficients. We emphasize that this numerical wavelet system is redundant, unlike the orthonormal wavelet basis used as the theoretical benchmark in \Cref{sub:comparing}. We use the stationary wavelet transform here to obtain a translation-invariant system like  digital shearlets.

Now we specify the analysis operator $V$ of our measurement frame. In both cases of the wavelet/shearlet sparsifying frames, we consider two-dimensional Fourier measurements, and thus we numerically model the application of $V$ to a two-dimensional signal with the two-dimensional Fourier transform \texttt{fft2}. 

\subsubsection{Sampling probability distribution}\label{numericsDistribution}

In our numerics, we choose our two-dimensional measurement parameters from a finite integer grid $[-M+1, M]^2 \cap \mathbb{Z}^2$ to evaluate the Fourier transform. The assumption of an even number of measurements in each dimension is theoretically arbitrary, but it will be relevant in the numerical experiments in \Cref{numericsPhaseDiagrams} because it is required by the stationary wavelet transform implementation of \textsc{MATLAB}'s Wavelet Toolbox \cite{WaveletToolbox}. For sampled measurements $\Omega \subseteq [-M+1, M]^2 \cap \mathbb{Z}^2$, we model the measurement sampling operator $P_\Omega$ as a binary mask on the grid $[-M+1, M]^2 \cap \mathbb{Z}^2$, constructed by sampling measurements with replacement according to a suitable probability distribution. For any complex valued digital image $g \in \mathbb{C}^{K \times K}$, \texttt{fft2(g)} again has dimension $K\times K$ and we need the binary mask of dimension $2M \times 2M$ modeling $P_\Omega$ to have compatible dimensions with $Vg$. Thus, for our specific choice of measurement analysis operator, we require the dimensional relation $K = 2M$.  

\Cref{corollary2.9}, see in particular \eqref{eq:muk-omegak} and \eqref{eq:nul-omegal}, suggests choosing a sampling distribution compatible with the local coherence between our chosen sparsifying and measurement frames.  In the context of shearlets and two-dimensional wavelets, using \Cref{waveletCoherenceRemark} and  \Cref{localCoherenceProp}, the bounds suggest that the relevant decay rates are $\|n\|_\infty^{-1}$ for Daubechies wavelets and $\|n\|_\infty^{-3/4}$ for shearlets as $n\to \infty$. Thus, the rate of decay for shearlets satisfies the local coherence upper bound required by \Cref{corollary2.9} for both two-dimensional Daubechies wavelets and cone-adapted shearlets in the context of two-dimensional Fourier measurements. This motivates choosing a probability density function $p\colon[-M+1, M]^2 \cap \mathbb{Z}^2 \to [0,1]$  given by
$$
p(n_1, n_2)=
\begin{cases}
\dfrac{1}{C}, & (n_1, n_2)=(0,0),\\
\dfrac{1}{C\|(n_1, n_2)\|_\infty^{3/2}}, & (n_1, n_2)\neq(0,0),
\end{cases}
$$
where the normalizing constant $C$ is chosen so that $p$ has total mass one.

After drawing $k$ points independently and identically from the grid $[-M+1, M]^2 \cap \mathbb{Z}^2$ according to $p$, the binary mask $P_\Omega$ is the projection defined by keeping only the entries corresponding to the sampled measurements. For the numerical experiments of \Cref{numericsPhaseDiagrams}, $\mathcal{H}$ is the space of digital images $\mathbb{C}^{K \times K}$. In particular, for $g \in \mathbb{C}^{K \times K}$, $Vg$ is implemented as the full array of measurements and $P_\Omega Vg$ is implemented by elementwise multiplication of the measurement array $Vg$ with this binary mask. Since the required coherence bounds for both sparsifying frames we chose are satisfied with this probability distribution, we use $p$ to sample Fourier measurements in both the wavelet and shearlet cases. We have also conducted experiments with the probability density corresponding to the wavelet coherence bounds; the results are very similar, and we omit them.

\subsubsection{Construction of ground truth data} \label{numericsFuncConstruction}

For the experiments, we construct the ground truth signals $g^\dagger$ as real-valued, compactly supported, piecewise regular functions on the unit square $[0,1]^2$, and then uniformly sample them on a fixed $K\times K$ grid to obtain the digital images $g^\dagger \in \mathbb{R}^{K\times K}$. Let us now describe the construction process of these images. Let $r=\frac{1}{12}$ and consider the $25$ uniformly placed points in $[0, 1]^2$:
$$
c_{p,q}
=
\left(\frac{q-\frac{1}{2}}{5},\,\frac{p-\frac{1}{2}}{5}\right),
\quad p,q\in\{1,\dots,5\}.
$$
These points will be the centers of disks of radius $r$. From these $25$ centers, we choose $1\leq n \leq 25$ of them uniformly at random without replacement and denote the selected centers by
$$
c_j = (a_j,b_j), \quad j=1,\dots,n.
$$
To each selected center we associate the disk,
$$
B_j = \left\{(x,y)\in[0,1]^2 : (x-a_j)^2+(y-b_j)^2 \leq r^2\right\}, \quad j = 1,\dots, n.
$$
Since the spacing between neighboring centers is $\frac{1}{5}$ and $2r=\frac{1}{6}$, no two disks intersect. This later allows us to control the length of the boundary of the support of $g^\dagger$.

For each $j=1,\dots,n$, we draw independent random coefficients,
$$
\alpha_j \sim \operatorname{Unif}[0.95, 1.05],
\quad
\beta_j,\gamma_j \sim \mathcal{N}(0, 0.05^2),
$$
and with $\chi$ denoting the characteristic function, define
\begin{equation}
g^\dagger(x,y)
=
\sum_{j = 1}^n
(\alpha_j + \beta_j\sin(6\pi x) + \gamma_j\cos(6\pi y))\,
\chi_{B_j}(x,y),
\quad (x, y) \in [0,1]^2. \label{groundTruthImage}
\end{equation}
By construction, $g^\dagger$ is compactly supported in the union of the selected disks and is piecewise regular, since it  is given by a smooth function on each disk $B_j$, while it vanishes outside of these disks. Finally, the digital image $g^\dagger_K \in\mathbb R^{K\times K}$ is obtained by sampling this function on the uniform grid 
\begin{equation}
    G_K = \left\{ \left(\frac{x}{K},\frac{y}{K}\right) : x, y = 0, \dots, K-1 \right\}
\subseteq [0,1]^2, \label{uniformEvaluationGrid}
\end{equation}
so we have $g^\dagger_K = g^\dagger|_{G_K}$. Two randomly generated functions with $n = 5$ and $n = 14$ using the formula \eqref{groundTruthImage}, evaluated on the grid $G_{256}$, are shown in \Cref{fig:piecewiseRegularExamples} (panels \subref{fig:n5PiecewiseRegular} and \subref{fig:n14PiecewiseRegular}, respectively).

\begin{figure}[t]
\centering

\begin{subfigure}[t]{0.47\textwidth}
    \centering
    \includegraphics[width=\linewidth]{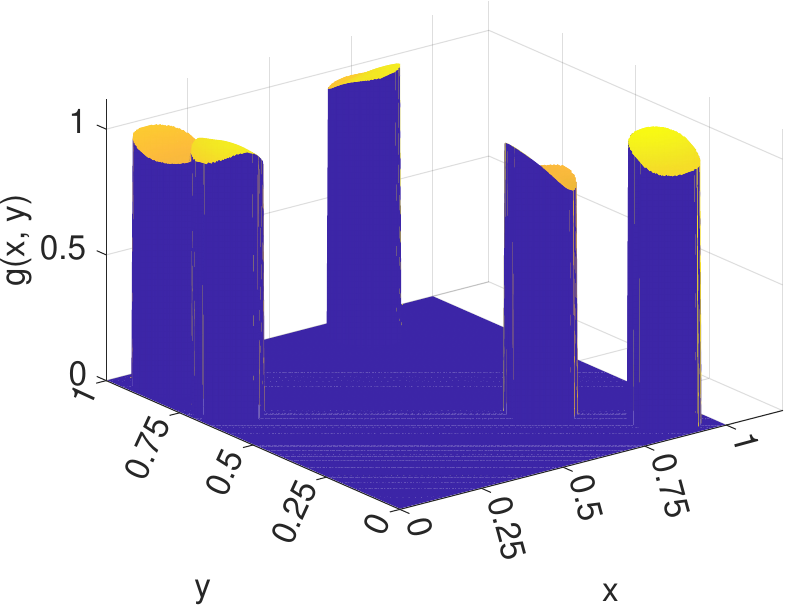}
    \caption{$n=5$ selected disks}
    \label{fig:n5PiecewiseRegular}
\end{subfigure}
\hfill
\begin{subfigure}[t]{0.47\textwidth}
    \centering
    \includegraphics[width=\linewidth]{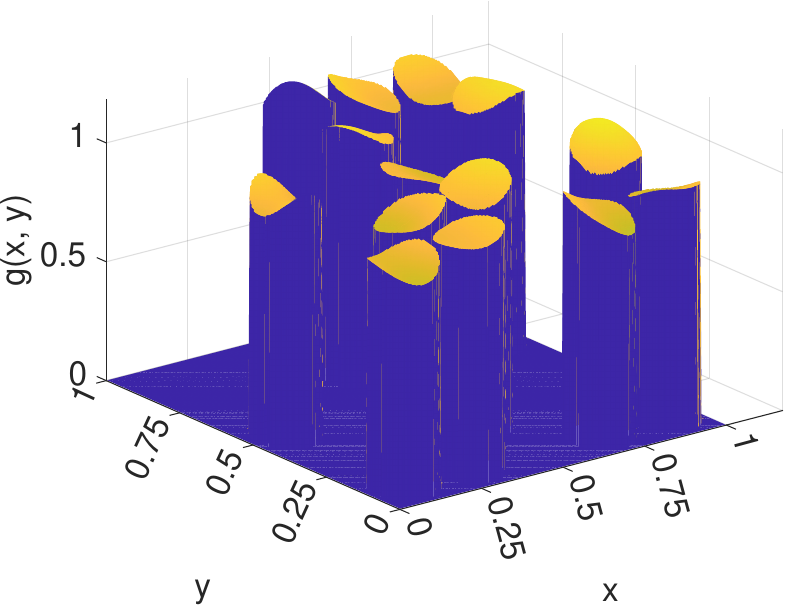}
    \caption{$n=14$ selected disks}
    \label{fig:n14PiecewiseRegular}
\end{subfigure}

\captionsetup{width=0.95\textwidth}
\caption[Examples of piecewise-regular ground-truth images]{Examples of synthetic ground-truth images $g^\dagger_{256}$ generated from \eqref{groundTruthImage} and sampled on the grid $G_{256}$. The support is the union of $n$ disjoint disks chosen from the fixed $5 \times 5$ array of possible centers.}
\label{fig:piecewiseRegularExamples}
\end{figure}

The test images are unions of characteristic functions of finitely many non-intersecting disks. They should be understood as multi-component cartoon-like images; the usual approximation estimates extend to this setting with constants depending on the number and the radius of the disks.

\subsection{The minimization problem} \label{numericsAnalysis}
For noisy ground truth data $y_\Omega = P_\Omega Vg^\dagger_K + \gamma_\Omega$ with $\|\gamma_\Omega\|_2 \leq \varepsilon_\Omega$, taking $\mathcal{H} = \mathbb{C}^{K \times K}$, we can write the analysis formulation \eqref{analysisFormulation} in the finite-dimensional setting as
\begin{equation}
    \min_{ g \in \mathbb{C}^{K \times K} } \|Dg\|_1
    \quad \text{subject to} \quad
    \|A g - y_\Omega\|_2 \leq \varepsilon_\Omega, \label{analysisFormulationRewritten}
\end{equation}
where
$$
A = P_\Omega V: \mathbb{C}^{K \times K} \to \mathbb{C}^{|\Omega|} .
$$

In our numerical implementations, the noise $\gamma_\Omega$ is generated with independently sampled Gaussian entries in $\mathcal{N}(0, \sigma^2)$ with $\sigma > 0$. Thus
$$
\mathbb{E}[\|\gamma_\Omega\|_2^2]
=
\sum_{i=1}^{|\Omega|} \sigma^2
=
|\Omega|\sigma^2.
$$
Accordingly, we take $\varepsilon_\Omega = \sigma\sqrt{|\Omega|}$ as a meaningful heuristic choice for the noise level in \eqref{analysisFormulationRewritten}. It is convenient to rewrite the constraint through an indicator function. Define the feasible set
$$
\mathcal{C}
=
\left\{
g \in \mathbb{C}^{K \times K} : \|Ag-y_\Omega\|_2 \le \varepsilon_\Omega
\right\},
$$
and let $\iota_{\mathcal{C}} \colon \mathbb{C}^{K \times K} \to [0,+\infty]$ be its indicator function,
$$
\iota_{\mathcal{C}}(g)
=
\begin{cases}
0, & g \in \mathcal{C},\\
+\infty, & g \notin \mathcal{C}.
\end{cases}
$$
Then the constrained problem \eqref{analysisFormulationRewritten} can be equivalently written as
\begin{equation}
    \min_{g \in \mathbb{C}^{K \times K}} \, \iota_{\mathcal{C}}(g) + \|Dg\|_1. \label{analysisFormUnlocbox}
\end{equation} 
This is precisely the form required by the forward-backward primal-dual solver in the \textsc{UNLocBoX} library \cite{Perraudin2014}, which solves minimization problems of the following type,
\begin{equation}
    \min_{g \in \mathbb{C}^{K \times K}} \, f_1(g) + f_2(Lg) + f_3(g), \label{primalDualSolver}
\end{equation}
where $L$ is a linear operator, $f_2$ is a convex function for which a proximal operator is available, and one of the terms $f_1$ or $f_3$ is a smooth convex function. In our setting, we take
$$
f_1(g) = \iota_{\mathcal{C}}(g),
\quad
f_2(z) = \|z\|_1,
\quad
L = D,
\quad
f_3(g)=0,
$$
which reduces \eqref{primalDualSolver} to \eqref{analysisFormUnlocbox}. 

The stopping rule for the minimization problem \eqref{primalDualSolver} is the one built into the forward-backward primal-dual solver. In our setting,  a \emph{primal variable} is the unknown signal $g$ to be reconstructed, whereas a \emph{dual variable} is an auxiliary variable introduced by the primal dual reformulation of the problem (see \cite{Komodakis2015}) to handle the term $f_2(Lg)$. More precisely, if $y_i^{(t)}$ denotes the $i$-th dual variable at iteration $t$, then the algorithm terminates for a tolerance value $\mathrm{tol}$ when the following condition is satisfied:
$$
\max_i \frac{\|y_i^{(t)}-y_i^{(t-1)}\|_2}{\|y_i^{(t)}\|_2} < \mathrm{tol}.
$$
In addition, the maximum number of iterations can be imposed. Thus, the algorithm is stopped as soon as either the above relative-change criterion is satisfied or the iteration count reaches this maximum. For our experiments, we use a relative error tolerance value of $10^{-3}$ and a maximum number of iterations of $10,000$.

\begin{remark} \label{numericsOperatorRemark1}
When applied to \eqref{primalDualSolver}, the \textsc{UNLocBoX} solver only requires routines for $A, A^*, D$ and $D^*$. Furthermore, evaluating $A$ and $A^*$ only requires the ability to evaluate $V$, $V^*$, the sampling operator $P_\Omega$ and the embedding $P_\Omega^*$ as seen in \Cref{alg:analysisForwardMap,alg:analysisAdjointMap}. Hence other measurement operators (for example the Radon transform or higher dimensional Fourier transforms) and even other sparsifying systems (for example other directional wavelets like curvelets) can be used by replacing these building blocks. 
\end{remark}

\begin{algorithm}[htbp]
  \caption{Implementation of the forward operator $A$}
  \label{alg:analysisForwardMap}
  \begin{algorithmic}[1]
    \State \textbf{Given:} Sampling set $\Omega$ and operator $V$.
    \Function{$w = A$}{$g$}
      \State \textbf{Input:} $g \in \mathbb{C}^{K \times K}$ \Comment{signal}
      \State $u \gets V g$ \Comment{apply measurement operator}
      \State $w \gets P_\Omega u$ \Comment{restrict to sampled indices}
      \State \textbf{Output:} $w \in \mathbb{C}^{|\Omega|}$
    \EndFunction
  \end{algorithmic}
\end{algorithm}

\begin{algorithm}[htbp]
  \caption{Implementation of the adjoint operator $A^*$}
  \label{alg:analysisAdjointMap}
  \begin{algorithmic}[1]
    \State \textbf{Given:} Sampling set $\Omega$ and operator $V^*$.
    
    \Function{$g = A^*$}{$w$}
      \State \textbf{Input:} $w \in \mathbb{C}^{|\Omega|}$ \Comment{data in measurement space}
      \State $u \gets P_\Omega^* w$ \Comment{zero-filling in $\Omega^c$}
      \State $g \gets V^* u$ \Comment{map the measurements back to signal space}
      \State \textbf{Output:} $g \in \mathbb{C}^{K \times K}$
    \EndFunction
  \end{algorithmic}
\end{algorithm}

As noted in \Cref{numericsOperatorRemark1}, the solver requires routines for the adjoints $D^*$, whereas the toolbox routines provide inverse transforms corresponding to $D^{-1}$. The code that evaluates $D^*$ for both cases is based on the fact that letting $N$ denote the number of sparsifying frame coefficients, $D^*$ is the synthesis operator of the sparsifying frame $\{\psi_i\}_{i = 1}^N$, so for coefficients $c = \{c_i\}_{i = 1}^N \in \C^N$,
$$
D^*c = \sum_{i = 1}^N c_i \psi_i.
$$

In our numerical experiments, the operator $V$ is implemented by (normalized) \texttt{fft2} and $P_\Omega$ is a binary sampling mask as described in \Cref{numericsDistribution}. The operator $D$ is either a two-dimensional wavelet transform or a two-dimensional shearlet transform. Then the forward map $A = P_\Omega V$ is realized by simply applying first the \texttt{fft2} and then masking the resulting Fourier measurements. The adjoint map $A^* = V^* P_\Omega^*$ instead is realized by first zero-filling the input on $\Omega^c$ in frequency space and then applying the inverse of \texttt{fft2}.

\subsection{Results}

\subsubsection{Nonlinear approximation rates}\label{subsub:nonlinearapproximation}

As introduced in \Cref{sub:CS-intro}, a standard way to quantify the approximation properties of a representation system is through its best $N$-term nonlinear approximation error:
$$
\sigma_N(f) =  \|f-f_N\|_2^2,
$$ 
see \Cref{def:nonlinear-approx}.
The associated nonlinear approximation rate is the asymptotic decay of $\sigma_N(f)$ as $N \to \infty$. For the class of two-dimensional piecewise regular functions considered here, the optimal benchmark is of order $\sigma_N(f) \asymp N^{-2}$ \cite{Devore1998}. Two-dimensional wavelets can reach a nonlinear approximation rate of $\mathcal{O}(N^{-1})$ \cite{Mallat1998}, while shearlets can perform nearly optimally (see Appendix~\ref{shearletAppendix} for additional details).

We now present numerical simulations comparing the nonlinear approximation properties of the  wavelet and shearlet frames introduced in $\S$\ref{numericsDigitalFrames}, with the class of real-valued functions introduced in $\S$\ref{numericsFuncConstruction}. In particular, we compare how the number of required coefficients $N$ for a prescribed nonlinear approximation error threshold depends on the geometric complexity of the underlying functions.

\begin{figure}
\centering
\begin{minipage}[b]{0.6\textwidth}
\centering
\includegraphics[width=\textwidth]{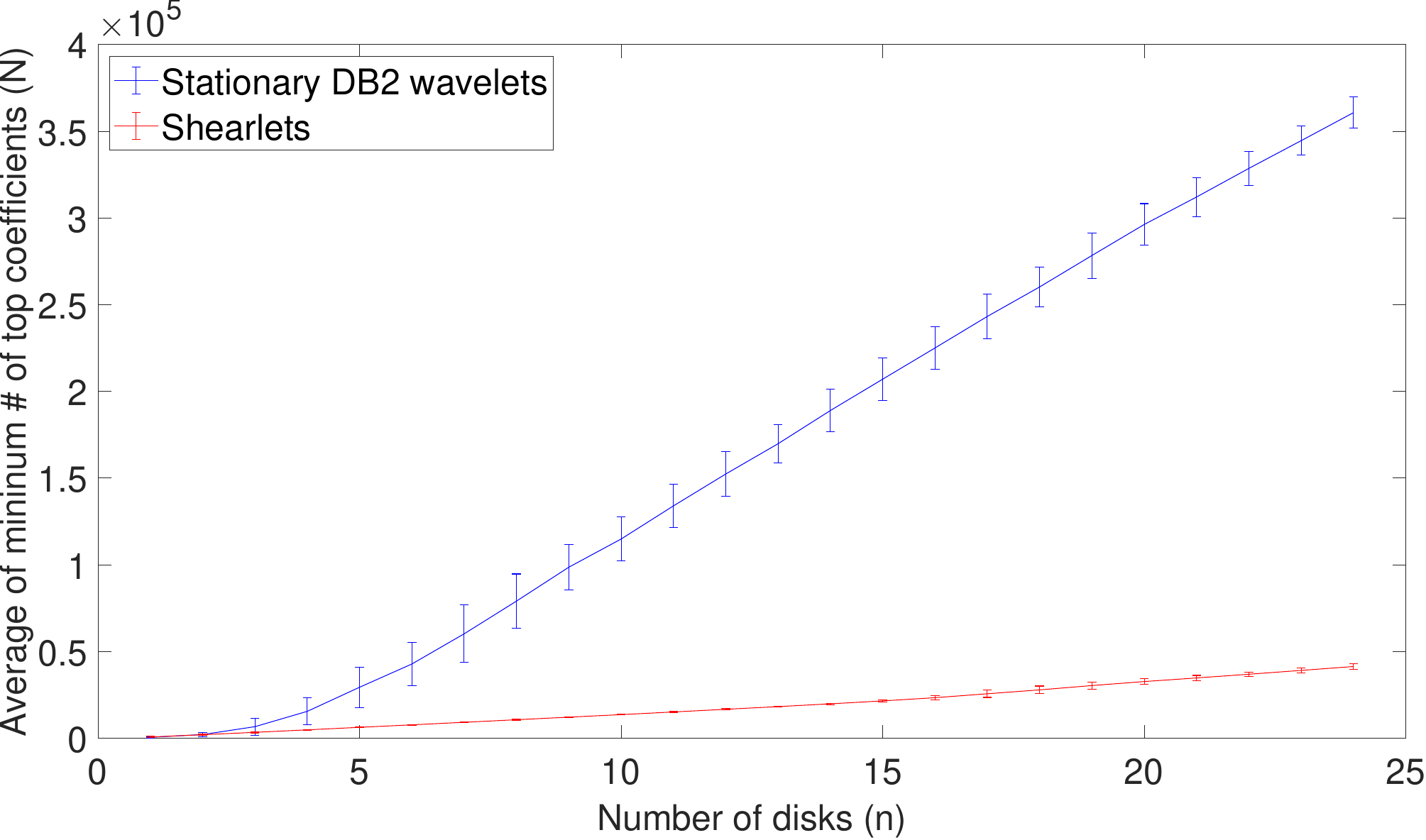}
\captionsetup{width=1.4\textwidth}
\caption{Minimum values of $N$ needed to achieve $\sigma_N(g^\dagger_{256}) \leq \varepsilon^2$ with $\varepsilon=25$, for test functions $g^\dagger_{256}$ with support boundary length $nL$ for wavelets and shearlets.}

\label{fig:coeffDataGraph}
\end{minipage}
\end{figure}

In \Cref{fig:coeffDataGraph}, we plot two curves, corresponding to wavelets and shearlets, that show how the minimum $N$'s needed to ensure $\sigma_N(f) \leq \varepsilon^2$ for $\varepsilon = 25$ depends on the number $n$ in \eqref{groundTruthImage}. More precisely, we denote the disk circumference $L = 2\pi r = \frac{\pi}{6}$ (see the construction in \Cref{numericsFuncConstruction}). For each support boundary length $nL$ with $1 \leq n \leq 24$ ($x$-axis)\footnote{The perimeter of the boundary is a common measure to quantify nonlinear approximation rates \cite{Mallat1998}, which motivates the use of $n$ in the $x$-axis.}, we generate $100$ random functions on $[0,1]^2$ of the form given in \eqref{groundTruthImage} and evaluate them on the uniform $256 \times 256$ grid $G_{256}$ as defined in \eqref{uniformEvaluationGrid}. On the $y$-axis, 
we plot the sample mean $\pm 1$ standard deviation of the minimum $N$'s such that $\sigma_N(g^\dagger_{256}) \leq \varepsilon^2$. As expected from the preceding theoretical discussion, shearlets yield a much sparser representation of this class of signals: in the notation of \Cref{sub:CS-intro}, we have $s_{\rm shear}\ll s_{\rm wave}$.

\subsubsection{Phase diagrams} \label{numericsPhaseDiagrams}

Phase diagram analysis provides a convenient way to summarize how the performance of a reconstruction method changes as the structural complexity of the unknown signals and the amount of available data vary. In compressed sensing, this point of view goes back to the seminal work of Donoho and Tanner \cite{Donoho2009, Donoho2010}, who showed that for Gaussian sensing matrices and $\ell^1$-based recovery, one observes a phase transition phenomenon in the diagrams. More specifically, above a critical curve recovery succeeds with high probability, whereas below it, recovery typically fails. These results make phase diagrams a natural numerical tool for visualizing the trade-off between complexity of our signals and sampling of measurements.

For our wavelet and shearlet settings, guided by the data displayed in \Cref{fig:coeffDataGraph}, we adopt this idea in an empirical form adapted to the analysis formulation. Rather than indexing the horizontal axis of the phase diagrams by coefficient sparsity, we parametrize the family of the piecewise regular functions on $[0, 1]^2$ constructed as in \eqref{groundTruthImage} by the length of the boundaries of their respective smooth compact supports. More precisely, the columns of the phase diagram correspond to the uniformly increasing support boundary lengths $nL$, while the rows correspond to subsampling percentages $\rho\%$, with $1 \leq \rho \leq 20$ and $1 \leq n \leq 24$. Notice that in the notation of \Cref{corollary2.9} and \Cref{numericsDistribution} the subsampling percentage $\rho$ corresponds to $\rho = 100 \cdot \frac{m}{K^2}$ since the number of sampled measurements is denoted by $m$ and our arrays of Fourier coefficients are of dimension $K \times K$. This allows us to better compare the performance of wavelets and shearlets. The range for the subsampling percentage $\rho$ is arbitrary and is chosen sufficiently wide so that the phase transition behavior can be clearly seen. 

At each grid point $(n, \rho)$, we show the probability of success in recovering images with $n$ disks with $\rho$\% subsampled measurements, as we now describe. We fix a positive integer $T$ and for each grid point $(n, \rho)$, we generate $T$ pairs of test functions with $n$ disks and measurement sampling masks with $\rho\%$ subsampling rate (with $\rho\%$ different Fourier measurements) and solve the corresponding compressed sensing problem as in \Cref{numericsAnalysis}. In all phase-diagram experiments we take \(T=6\). 

For $t = 1,\dots,T,$ let $g^\dagger_{(t,n,\rho)}$ denote the $t$-th, $256 \times 256$ ground truth digital image corresponding to the parameter pair $(n, \rho),$ and let $\bar g_{(t,n,\rho)}$ denote the associated reconstructed $256 \times 256$ digital image. We define the absolute and relative reconstruction errors by
$$
E_{\mathrm{abs}}^{(t,n,\rho)}
=
\bigl\|g^\dagger_{(t,n,\rho)}-\bar g_{(t,n,\rho)}\bigr\|_2,
$$
and
$$
E_{\mathrm{rel}}^{(t,n,\rho)}
=
\frac{\bigl\|g^\dagger_{(t,n,\rho)}-\bar g_{(t,n,\rho)}\bigr\|_2}
{\bigl\|g^\dagger_{(t,n,\rho)}\bigr\|_2},
$$
respectively.

Given an absolute error threshold $\varepsilon_a>0$, we declare a reconstruction successful whenever the corresponding absolute reconstruction error is at most $\varepsilon_a$. Thus, for the absolute-error phase diagram, the value shown at the grid point $(n, \rho)$ is
$$
P_{\mathrm{abs}}(n,\rho)
= \frac{1}{T}
\sum_{t=1}^T
\chi_{\{E_{\mathrm{abs}}^{(t,n,\rho)}\leq \varepsilon_a\}}.
$$
Similarly, given a relative error threshold $\varepsilon_r>0$, we declare a reconstruction successful whenever the corresponding relative reconstruction error is at most $\varepsilon_r$. Thus, for the relative-error phase diagram, the displayed value at the grid point $(n, \rho)$ is
$$
P_{\mathrm{rel}}(n,\rho)
=
\frac{1}{T}
\sum_{t=1}^T
\chi_{\{E_{\mathrm{rel}}^{(t,n,\rho)}\leq \varepsilon_r\}}.
$$
These two quantities are the proportions of successful reconstructions (where success depends on $\varepsilon_a$ and $\varepsilon_r$) within the $T$ experiments, that is, the empirical probabilities obtained by assigning equal weight to each of the $T$ trials.

An additional advantage of this framework is that each reconstruction problem associated with a parameter triple $(t,n,\rho)$ can be solved independently of all others. Thus, the computation of the full phase diagram is a highly parallelizable task, which makes it well suited to large-scale numerical experimentation. In the following, we apply this general phase diagram methodology to our two sparsifying frames, namely wavelets and shearlets, with subsampled Fourier measurements, as described in \Cref{numericsDigitalFrames,numericsDistribution}.

We obtain the phase diagrams for wavelet/shearlet sparsifying frames with an absolute error threshold of $\varepsilon_a = 25$ in \Cref{fig:analysisWavAbsPhaseDiag,fig:analysisShearAbsPhaseDiag} and with a relative error threshold of $\varepsilon_r = 0.15$ in \Cref{fig:analysisWavRelPhaseDiag,fig:analysisShearRelPhaseDiag}, respectively.
Despite the substantially smaller sparsity with shearlets (see $\S$\ref{subsub:nonlinearapproximation} and \Cref{fig:coeffDataGraph}), these phase diagrams show only a minor advantage when shearlets are used rather than wavelets for this dataset (roughly a factor of 2 in the sample complexity). This suggests that the large advantage in terms of sparsity provided by shearlets does not translate into a comparable gain in terms of sample complexity.

\begin{figure}[t]
\centering

\begin{subfigure}[t]{0.49\textwidth}
    \centering
    \includegraphics[width=\linewidth]{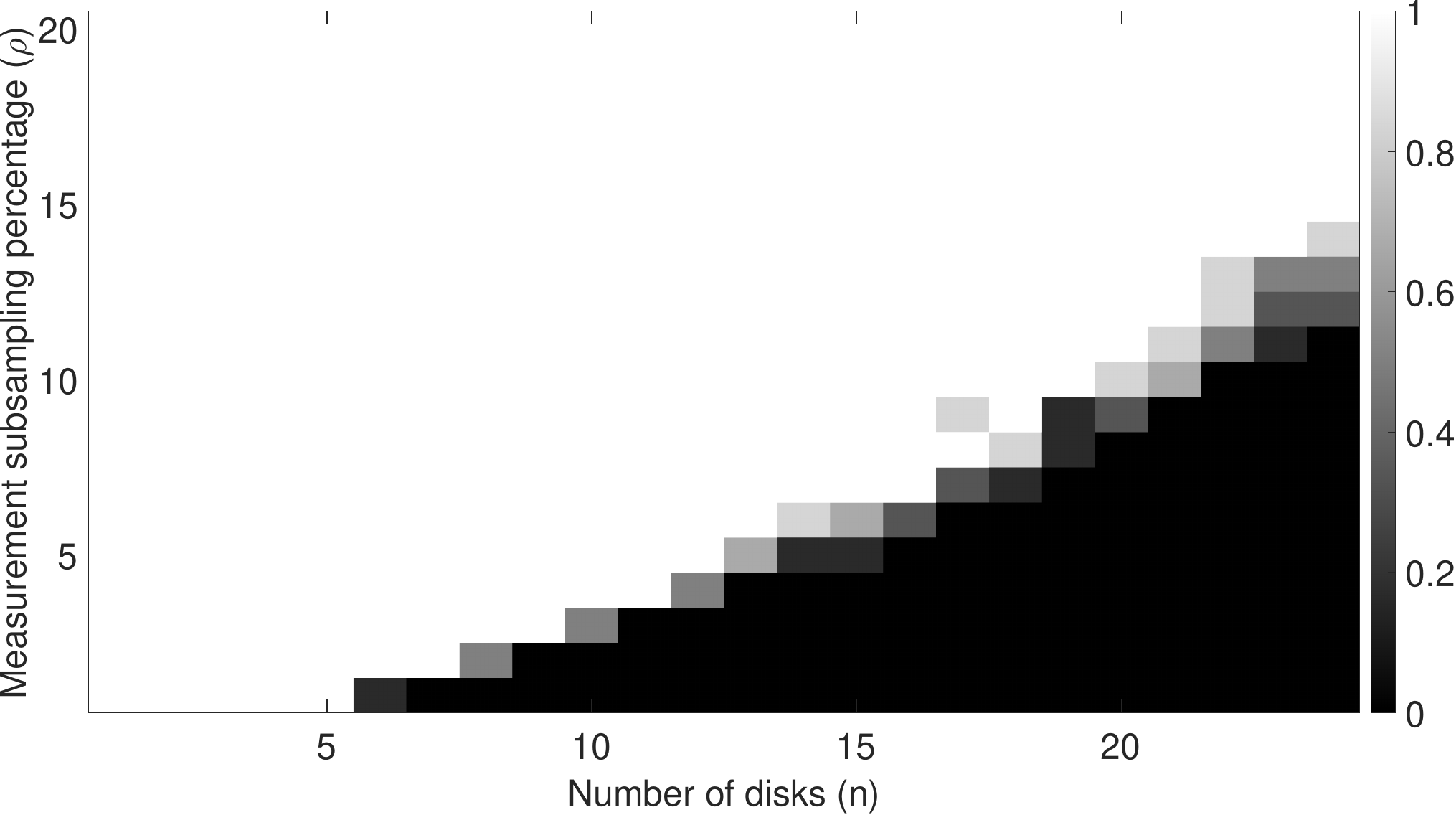}
    \caption{Absolute-error criterion $\varepsilon_a = 25$, stationary Daubechies--2 (DB2) wavelets with four levels.}
    \label{fig:analysisWavAbsPhaseDiag}
\end{subfigure}
\hfill
\begin{subfigure}[t]{0.49\textwidth}
    \centering
    \includegraphics[width=\linewidth]{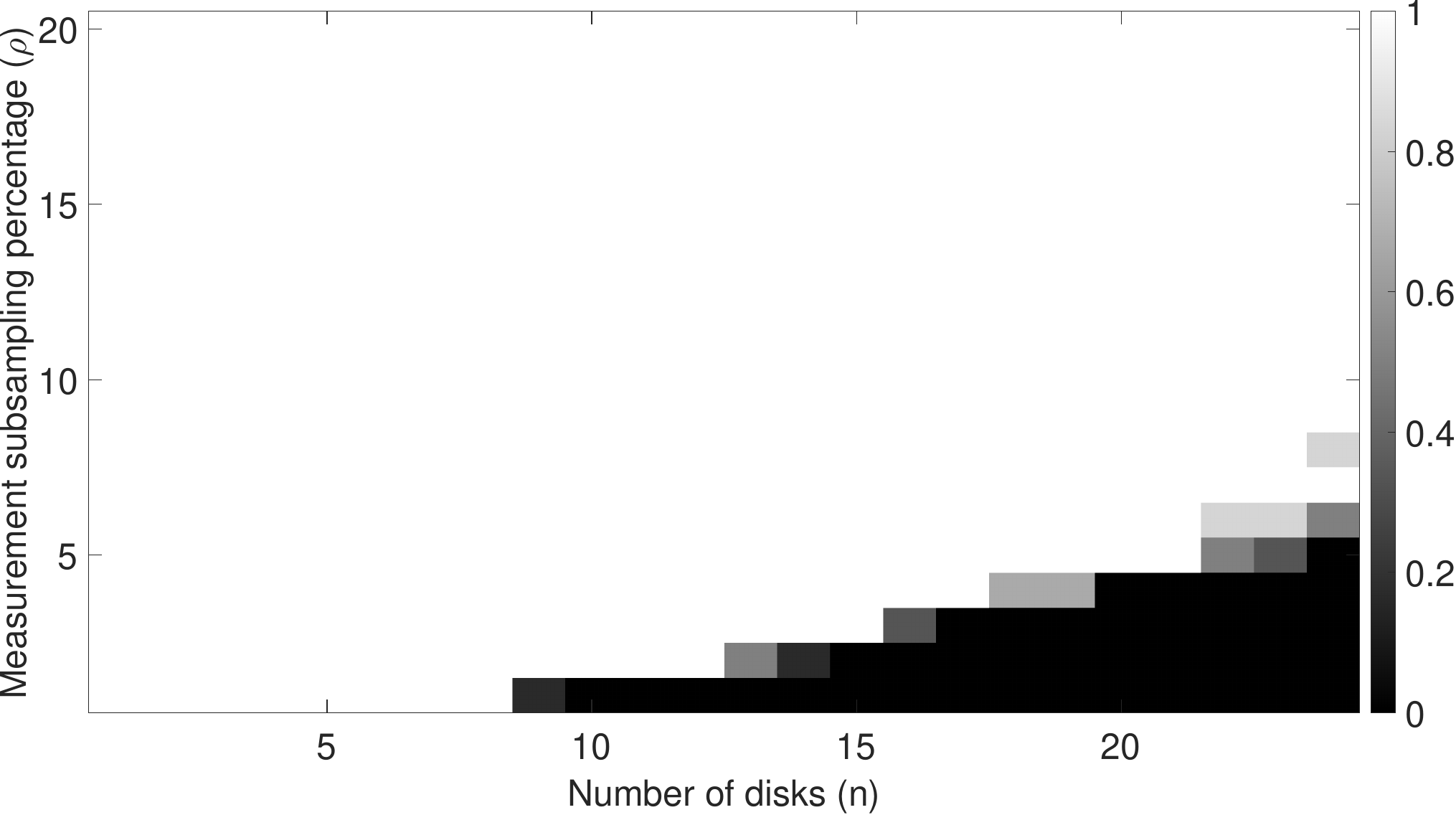}
    \caption{Absolute-error criterion $\varepsilon_a = 25$, shearlets with scales up to 2.}
    \label{fig:analysisShearAbsPhaseDiag}
\end{subfigure}

\vspace{0.8em}

\begin{subfigure}[t]{0.49\textwidth}
    \centering
    \includegraphics[width=\linewidth]{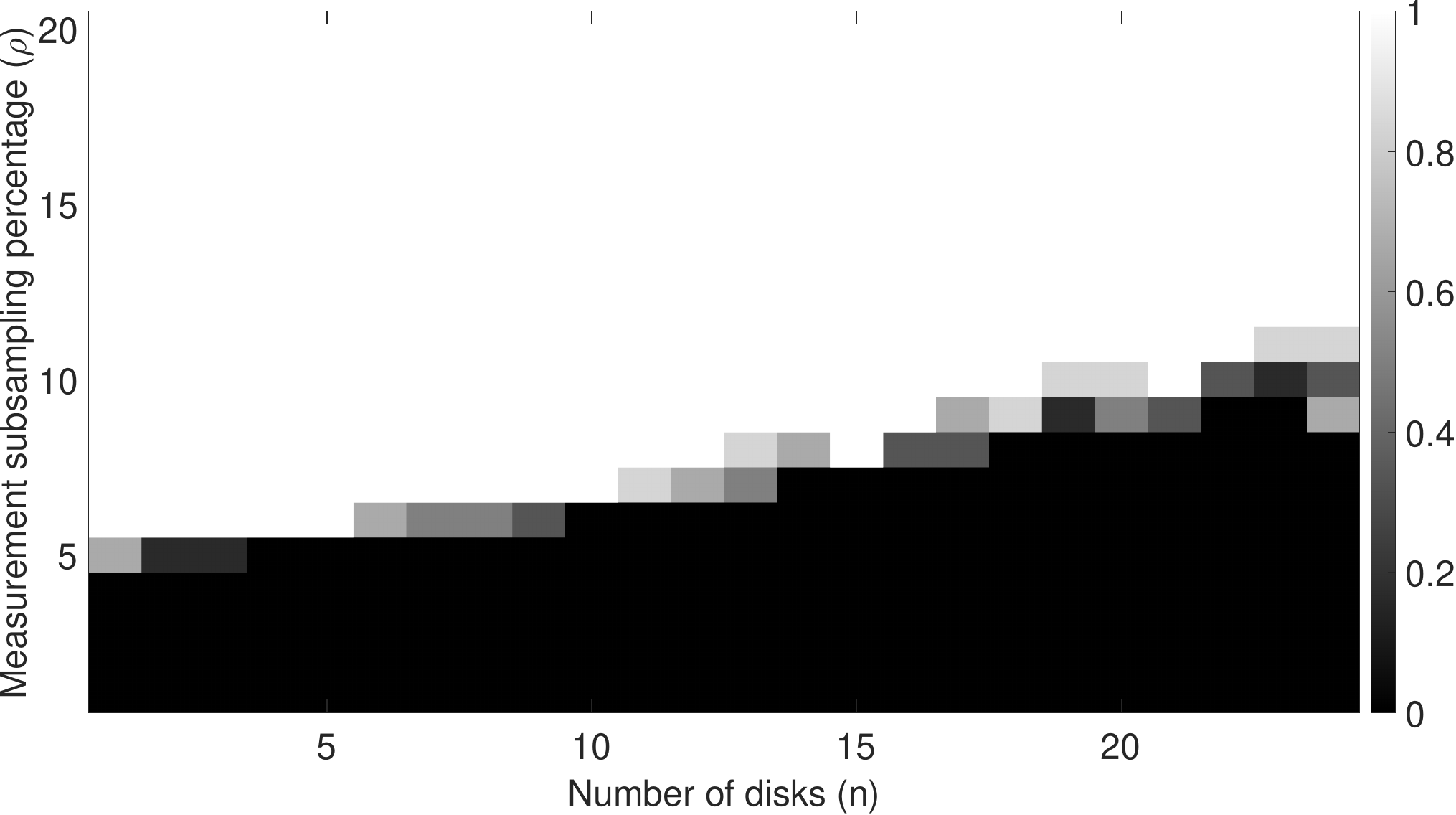}
    \caption{Relative-error criterion $\varepsilon_r = 0.15$, stationary Daubechies--2 (DB2) wavelets with four levels.}
    \label{fig:analysisWavRelPhaseDiag}
\end{subfigure}
\hfill
\begin{subfigure}[t]{0.49\textwidth}
    \centering
    \includegraphics[width=\linewidth]{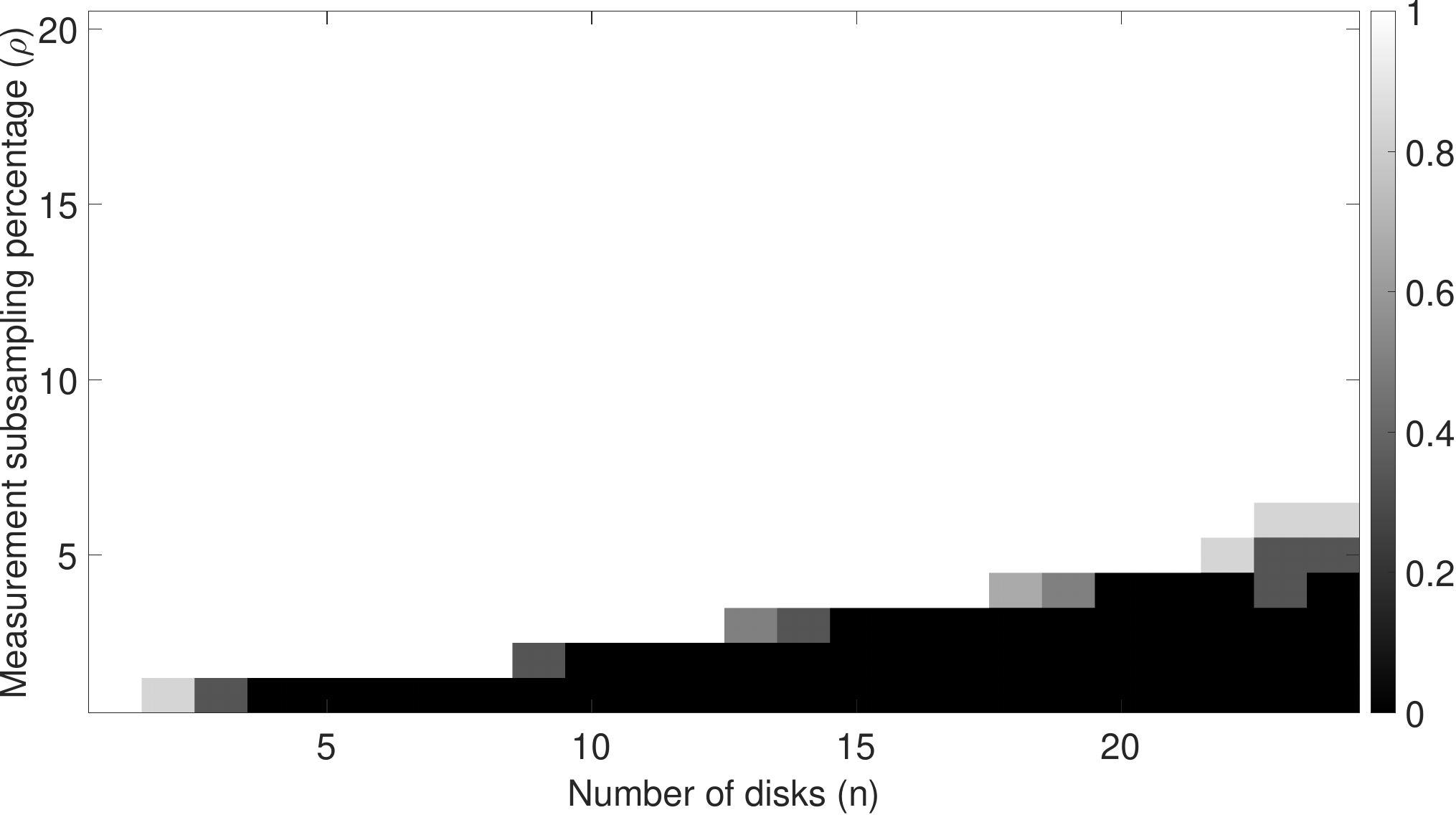}
    \caption{Relative-error criterion $\varepsilon_r = 0.15$, shearlets with scales up to 2.}
    \label{fig:analysisShearRelPhaseDiag}
\end{subfigure}

\captionsetup{width=0.95\textwidth}
\caption[Phase diagrams for wavelet and shearlet analysis reconstruction]{Phase diagrams for the empirical success probability of the analysis $\ell^1$ reconstruction from subsampled Fourier measurements. In each panel, the horizontal coordinate is the support-boundary length parameter $nL$ of the test-image class and the vertical coordinate is the sampling percentage $\rho$; the color at $(n,\rho)$ is the proportion of successful reconstructions over the $T$ independent trials. The top row reports $P_{\mathrm{abs}}(n,\rho)$ with absolute-error threshold $\varepsilon_a=25$, while the bottom row reports $P_{\mathrm{rel}}(n,\rho)$ with relative-error threshold $\varepsilon_r=0.15$.}
\label{fig:phaseDiagrams}
\end{figure}

\section{Conclusions}

In this paper, we investigated the sample complexity of compressed sensing when using shearlets as the sparsifying frame, systematically comparing their performance against standard wavelets. From a theoretical perspective, we detailed the obstacles that arise when applying known recovery theorems to shearlet systems; specifically, we showed that bounding the localization factor to control frame redundancy yields suboptimal bounds that can be exponentially large. As a result, deriving an explicit, optimal sample complexity estimate for shearlets appears out of reach using the currently available estimates. Furthermore, even  when the localization factor is assumed to be bounded, we show that the coherence properties satisfied by shearlets yield sample complexity estimates that do not improve on the corresponding wavelet estimates by more than logarithmic factors. 

To complement this analysis, we conducted a comprehensive numerical investigation on a dataset of piecewise regular, cartoon-like images. While our empirical results confirmed that shearlets offer substantially better nonlinear approximation properties and require substantially fewer coefficients for accurate representation compared to wavelets, our phase diagram analysis revealed that this better sparsity does not translate into a proportional reduction in sample complexity. Ultimately, the number of subsampled Fourier measurements required for successful signal recovery remains comparable for both representation systems.

These findings indicate several open directions for future research. First, it is natural to wonder whether an approach based on multilevel sparsity \cite{Poon2017} could allow us to improve the sample complexity estimates. Second, additional numerical simulations should be conducted in the infinite-dimensional setting to determine if theoretical obstacles, for example the balancing property, impact practical recovery rates differently than in the finite digital implementations evaluated here. Finally, while our current numerical experiments used a simple dataset of digital images to help make a clear and systematic comparison between the two sparsifying frames, future investigations should assess whether these sample complexity behaviors persist when applied to more realistic or less structured datasets.

\section*{Acknowledgements}
The authors would like to thank Hartmut Führ, Demetrio Labate, Philipp Petersen, and Felix Voigtlaender for useful insights related to the linear independence of shearlets. It is a pleasure to thank Filippo De Mari for multiple discussions on the localization factor for shearlets.

Co-funded by the European Union (ERC, SAMPDE, 101041040). Views and opinions expressed are  those of the authors only and do not necessarily reflect those of the European Union or the European Research Council. Neither the European Union nor the granting authority can be held responsible for them. 
The research was supported in part by the MIUR Excellence Department Project awarded to Dipartimento di Matematica, Università di Genova, CUP D33C23001110001.
 The authors are members of the “Gruppo Nazionale per l’Analisi Matematica, la Probabilità e le loro Applicazioni” of the “Istituto Nazionale di Alta Matematica”.

\appendix

\section{Infinite-dimensional compressed sensing} \label{metaResultAppendix}

While \Cref{corollary2.9} is a sufficient model for our specific numerical experiments, results for infinite sparsifying and measurement frames also exist. The assumption of \cite[Theorem 4.4]{Poon2017} on the operators $V$ and $D$ is that they are isometries, which is equivalent to the frames $\{\varphi_l\}_{l \in L}$ and $\{\psi_i\}_{i \in I}$ both being Parseval. This is suitable for our case because cone-adapted shearlets can form a Parseval frame  with a certain class of bandlimited shearlets \cite{Guo2007}. Our measurements are Fourier measurements, corresponding to $\{\varphi_l\}_{l \in L}$ being an orthonormal basis of complex exponentials. The probability distribution for sampling the measurements in the case of \cite[Theorem 4.4]{Poon2017} considers a multilevel sampling scheme where the wavelet coefficients are divided into levels with sparsities $s_k$ in each level and the Fourier measurements are divided into corresponding frequency bands with $m_k$ measurements in the $k$-th band. A result given in \cite[Corollary 1]{Alberti2021} also considers nonuniform distributions. 

In this subsection, we discuss a simplified version of the results given in \cite[Corollary 1]{Alberti2021} and \cite[Theorem 4.4]{Poon2017} to show the additional difficulties of the infinite-dimensional setting, especially in comparison to the  assumptions of \Cref{corollary2.9}. As we are interested in the case where both the measurement and the sparsifying frames are Parseval, for simplicity we state the balancing property taken from \cite{Alberti2021} restricted to this case. We consider the setting of \Cref{sec:frametheory}, with $I = L = \N$.

The first new ingredient of the infinite-dimensional setting is the balancing property.

\begin{definition} \label{balancingProperty}
    Let $s, N \in \N$ be such that $2 \leq s \leq N$. We say that $M \in \N$ satisfies the \emph{balancing property} with respect to $V$, $D$, $N$ and $s$ if for all $\Delta \subseteq \{1, \dots, N\}$ with $|\Delta| = s$ we have:
    $$\|P_\mathcal{W}V^*P_M^\perp VP_\mathcal{W}\|_{\mathcal{H} \to \mathcal{H}} \leq \frac{1}{8\sqrt{\log s}},$$
    and
    $$\|P_\Delta^\perp DP_\mathcal{W}^\perp V^*P_M V P_\mathcal{W}\|_{\mathcal{H} \to \ell^\infty} \leq \frac{1}{14 \sqrt{s }},$$
    where $\mathcal{W}:= \operatorname{Im}(D^*P_\Delta)$.
\end{definition}

For the more general case of multilevel sparsity and even for the case where the shearlet frame is not necessarily tight, certain methods to show that they can still satisfy \Cref{balancingProperty} have been developed in \cite[Section 3]{Ma2016thesis}. The key step in these results is to show that the shearlet frame in question is formed by what are called $\alpha$-\emph{shearlet molecules} \cite{Grohs2016} which guarantees that the shearlets have suitable regularity and decay properties. 

Although the results in \cite{Ma2016thesis} focus on a specific frame of compactly supported shearlets that do not form a tight frame, bandlimited shearlets that form Parseval frames (see \cite{Guo2013} and Appendix~\ref{shearletAppendix}) are also $\alpha$-shearlet molecules as shown in \cite[Proposition 3.11 (ii)]{Grohs2016}. We need the following notations:
$$B_{s, N} = \max\{B_\Delta\,:\, \Delta \subseteq \{1,\dots,N\}, \,  |\Delta| \leq s\}$$
where
$$B_\Delta = \max\{\|D P_\mathcal{W}^\perp D^*\|_{\ell^\infty \to \ell^\infty}, \, 1\}.$$
Further, for $\alpha \in (0, 1]$, let $\Tilde{N}(\alpha)$ be the smallest integer larger than or equal to $N$ such that  
$$\|P_M V D^{*}e_i\|_2 < \alpha,\quad \forall i>\Tilde{N}(\alpha).$$

\begin{remark}
    $B_{s, N}$ depends on the sparsifying frame and it is known that $\{\psi_i\}_{i\in I}$ being Parseval does not guarantee that $B_{s, N}$ is finite unless it is an orthonormal basis \cite[Example 2]{Alberti2021}. For $\Tilde{N}$, this function can be bounded if the sparsifying frame is an orthonormal basis and $\sup_{l \leq M}|\langle \psi_i, \varphi_l\rangle|$ decays sufficiently in $i$. Also, if $I = \{1, \dots, N\}$ is finite, we simply have $\Tilde{N}(\alpha) = N$ for all $\alpha \in (0, 1]$ \cite[Remark 5]{Alberti2021}. However, the types of frames of shearlets that are Parseval and have good nonlinear approximation properties like the one given in Appendix~\ref{shearletAppendix}, neither have finitely many elements nor are they orthonormal bases.
\end{remark}

Although we do not provide  full statements for the infinite-dimensional results from \cite{Poon2017, Alberti2021}, we draw a quick comparison between these results and \Cref{corollary2.9}. The assumptions on the coherence, on the localization factor and on the non-uniform sampling of the measurements are similar, as well as the optimization problem and the recovery estimate. The first main difference in the infinite-dimensional context is the need to truncate the available samples with the balancing property, yielding $M$ scalar measurements to sample from. The second main difference is in the sample complexity: the bound for the infinite-dimensional results reads
$$m \geq C_1\eta_{s, D}^2 \|\omega\|_2^2s B_{s,N}^2 \log\bigg(\Tilde{N}\bigg(\frac{C_2}{Mm\|\omega\|_2^2\sqrt{s}} \bigg)\bigg).$$
If we compare this with \eqref{eq:cor:m}, we see that the main difference lies in the presence of $B_{s,N}^2$ and of $\tilde N$ in the logarithmic factor. Finding suitable upper bounds of these quantities for shearlets would not be straightforward.

\section{Introduction to parabolic molecules} \label{parabolicMoleculesAppendix}

As also mentioned in Appendix~\ref{metaResultAppendix}, we are interested in certain classes of shearlets, which are also $\alpha$-molecules. These molecules are also referred to as \emph{parabolic molecules} \cite{Grohs2013}. Decay properties of parabolic molecules are useful for our setting, so in this subsection we describe their setup following \cite[Section 2]{Grohs2013}. We define the parameter space
$$
\mathbb{P} := \mathbb{R}_+ \times \mathbb{T} \times \mathbb{R}^2,
$$
where a point $p = (s,\theta,x)\in\mathbb{P}$ describes a scale $2^s$, an orientation $\theta$ on the torus $\mathbb{T}$, and a location $x\in\R^2$. Parabolic molecules are frames of functions $(m_\lambda)_{\lambda\in\Lambda}$, where each $m_\lambda \in L^2(\mathbb{R}^2)$ satisfies some additional regularity properties and each molecule is associated to a unique point in $\mathbb{P}$.

\begin{definition} \label{parametrizationDef}
\begin{enumerate}
    \item A \emph{parametrization} consists of a pair $(\Lambda,\Phi_\Lambda)$ where $\Lambda$ is an index set and $\Phi_\Lambda$ is a mapping
    $$
    \Phi_\Lambda :
    \Lambda\ni\lambda\mapsto (s_\lambda,\theta_\lambda,x_\lambda)\in\mathbb{P}
    $$
    that associates to each $\lambda\in\Lambda$ a scale $s_\lambda$, a direction $\theta_\lambda$ and a location $x_\lambda\in\mathbb{R}^2$.

\item A \emph{canonical parametrization} is the pair $(\Lambda^0, \Phi^0)$ given by
    $$
    \Lambda^0 := \left\{(j,l,k)\in\mathbb{Z}^4 : j\geq 0,\; l=-2^{\left\lfloor \frac{j}{2}\right\rfloor-1},\ldots,2^{\left\lfloor \frac{j}{2}\right\rfloor-1}\right\},
    $$
    where for $\lambda=(j,l,k)$ we define $\Phi^0(\lambda):=(s_\lambda,\theta_\lambda,x_\lambda)$ with $s_\lambda:=j$, $\theta_\lambda:=l\,2^{-\left\lfloor j/2\right\rfloor}\pi$, $x_\lambda:=R_{-\theta_\lambda}D_{2^{-s_\lambda}}k$,  $D_a:=\operatorname{diag}(a,\sqrt{a})$ associated with $a>0$ and
    $$
    R_\theta =
    \begin{pmatrix}
    \cos(\theta) & -\sin(\theta)\\
    \sin(\theta) & \cos(\theta)
    \end{pmatrix}.
    $$
\end{enumerate}
\end{definition}

\begin{definition} \label{parabolicMoleculeDef}
    Let $(\Lambda, \Phi_\Lambda)$ be a parametrization. A family $(m_\lambda)_{\lambda\in\Lambda}$ is called a \emph{family of parabolic molecules} of order $(R,M,N_1,N_2)$ if it can be written as
    $$    m_\lambda(x)=2^{3s_\lambda/4}a^{(\lambda)}\bigl(D_{2^{s_\lambda}}R_{\theta_\lambda}(x-x_\lambda)\bigr)
    $$
    and if there exists $C > 0$ such that for all $\lambda \in \Lambda$ we have,
    $$
    \left|\partial^\beta \hat{a}^{(\lambda)}(\xi)\right|
    \leq
    C\min\left(1,2^{-s_\lambda}+|\xi_1|+2^{-s_\lambda/2}|\xi_2|\right)^M(1 + \xi_1^2 + \xi_2^2) ^{-N_1/2}(1 + \xi_2^2)^{-N_2/2},\quad \xi\in\R^2,
    $$
    for all $|\beta|\leq R$. 
\end{definition}

To continue, we need a notion of distance between the indices of parabolic molecules. The parameter space $\mathbb{P}$ can be equipped with a notion of pseudo-distance, see \cite{Smith1998}.

\begin{definition} \label{indexDistDef}
    Let $(\Lambda, \Phi_\Lambda)$ and $(\Gamma, \Phi_\Gamma)$ be two parametrizations. For two indices $\lambda \in \Lambda$ and $\gamma \in \Gamma$, we define the \emph{index distance}
    $$\omega(\Phi_\Lambda(\lambda), \Phi_\Gamma(\gamma)) = 2^{|s_\lambda - s_\gamma|}(1 + 2^{\min(s_\lambda, s_\gamma)}d(\Phi_\Lambda(\lambda), \Phi_\Gamma(\gamma))),$$
    where
    $$d(\Phi_\Lambda(\lambda), \Phi_\Gamma(\gamma)) = |\theta_\lambda - \theta_\gamma|^2 + |x_\lambda - x_\gamma|^2 + |\langle e_\lambda, x_\lambda - x_\gamma\rangle|$$
    and  $e_\lambda = (\cos(\theta_\lambda), \sin(\theta_\lambda))^T$.
\end{definition}

We need the following summability result associated with the canonical parametrization.

\begin{lemma}[{\cite[Lemma 2.8]{Grohs2013}}]\label{LadmissibleLemma}
    Let $L > 2$. Then the canonical parametrization $(\Lambda^0, \Phi^0)$ is $L$-\emph{admissible} in the sense that,
    $$\sup_{\lambda \in \Lambda^0} \sum_{\mu \in \Lambda^0} \omega(\Phi^0(\lambda), \Phi^0(\mu))^{-L} < \infty.$$
\end{lemma}

Finally, we can state the analogue of intrinsic localization for general frames for parabolic molecules.

\begin{theorem}[{\cite[Theorem 2.9]{Grohs2013}}]\label{moleculeIntrinsicLocTheorem}
    Let $(\Lambda, \Phi_\Lambda)$ and $(\Gamma, \Phi_\Gamma)$ be two parametrizations. Let $(m_\lambda)_{\lambda \in \Lambda}$ and $(p_\gamma)_{\gamma \in \Gamma}$ be two families of parabolic molecules of order $(R, M, N_1, N_2)$ with
    $$R \geq 2L, \quad M > 4L - \frac{5}{4}, \quad N_1 \geq 2L + \frac{3}{4}, \quad N_2 \geq 2L$$
    for some $L >0 $.
    Then, there exists some $C > 0$ such that for every $\lambda \in \Lambda$ and $\gamma \in \Gamma$,
    $$|\langle m_\lambda, p_\gamma\rangle| \leq C\omega(\Phi_\Lambda(\lambda), \Phi_\Gamma(\gamma))^{-L}.$$
\end{theorem}

\section{Introduction to shearlets} \label{shearletAppendix}

In this section, we first define feasible cone-adapted shearlet generators and then state classical results on nonlinear approximation rates for commonly used frames of compactly supported shearlets and bandlimited shearlets. These shearlets are essentially supported in frequency in  the horizontal and vertical cones defined by:
$$
\mathcal{P}_1 = \left\{ (\xi_1,\xi_2)\in\mathbb{R}^2 : |\xi_2|\leq |\xi_1| \right\},
\quad
\mathcal{P}_2 = \left\{ (\xi_1,\xi_2)\in\mathbb{R}^2 : |\xi_2|>|\xi_1| \right\}.
$$

The feasible generators or the so-called mother shearlets of a cone-adapted discrete shearlet system as in \Cref{shearletsInSpace} need to have sufficient decay conditions in frequency that allow their essential supports to be concentrated on coronae that tile the frequency plane (see \cite[Section 2.3.2]{Kutyniok2012-2}) so that the systems can form frames. We slightly modify the condition $\alpha>\gamma>3$ of \cite[Definition 2.6]{Kutyniok2012-2} on the decay parameters $\alpha$ and $\gamma$ for the purposes of the proof of \Cref{localCoherenceProp}.

\begin{definition} \label{coneAdaptedGenerators}
Fix parameters $\alpha > \gamma \geq \frac{3}{4}$ and $q>q'>0$, $q > r>0$. Let the dilation and shearing matrices be 
$$
A_{(1)} =
\begin{pmatrix}
2 & 0 \\
0 & \sqrt{2}
\end{pmatrix},
\quad
B_{(1)} =
\begin{pmatrix}
1 & 1 \\
0 & 1
\end{pmatrix},
\quad
A_{(2)} =
\begin{pmatrix}
\sqrt{2} & 0 \\
0 & 2
\end{pmatrix},
\quad
B_{(2)} =
\begin{pmatrix}
1 & 0 \\
1 & 1
\end{pmatrix}.
$$
For $\Phi, \psi^{(1)}, \psi^{(2)}, \psi^{(1, b)}, \psi^{(2, b)} \in L^2(\mathbb{R}^2)$ we define the following generators and associated shearlets.
\begin{enumerate}
\item The function $\Phi$ is a \emph{feasible coarse-scale generator} if $\widehat\Phi(0,0)=1$ and for all $\xi=(\xi_1,\xi_2)\in\mathbb R^2 \setminus \{(0, 0)\}$,
$$
|\widehat\Phi(\xi_1,\xi_2)|
\le
\min\{1,|r\xi_1|^{-\gamma}\}\,\min\{1,|r\xi_2|^{-\gamma}\}.
$$
The associated coarse-scale shearlets are
$$
(\widetilde\psi_{-1,k})^{\wedge}(\xi)=\widehat\Phi(\xi)\,e^{-2\pi i\langle \xi,k\rangle}.
$$

\item The functions $\psi^{(1)}$ and $\psi^{(2)}$ are  \emph{feasible horizontal and vertical generators} if $\widehat{\psi}^{(1)}(0,0)=\widehat{\psi}^{(2)}(0,0)=0$ and for all $\xi=(\xi_1,\xi_2)\in\mathbb R^2 \setminus \{(0, 0)\}$,
$$
|\widehat{\psi}^{(1)}(\xi_1,\xi_2)|
\le
\min\{1,|q\xi_1|^{\alpha}\}\,
\min\{1,|q'\xi_1|^{-\gamma}\}\,
\min\{1,|r\xi_2|^{-\gamma}\},
$$
$$
|\widehat{\psi}^{(2)}(\xi_1,\xi_2)|
\le
\min\{1,|q\xi_2|^{\alpha}\}\,
\min\{1,|q'\xi_2|^{-\gamma}\}\,
\min\{1,|r\xi_1|^{-\gamma}\}.
$$
For $j\ge 0$, $|\ell|< \lceil2^{j/2}\rceil$, $k\in\mathbb Z^2$, define interior cone-adapted shearlets as
$$
(\psi^{(1)}_{j,\ell,k})^{\wedge}(\xi)
=
2^{-\frac{3}{4}j}\,
\widehat{\psi}^{(1)}\bigl(\xi A_{(1)}^{-j}B_{(1)}^{-\ell}\bigr)\,
e^{-2\pi i\langle \xi A_{(1)}^{-j}B_{(1)}^{-\ell},\,k\rangle},
$$
$$
(\psi^{(2)}_{j,\ell,k})^{\wedge}(\xi)
=
2^{-\frac{3}{4}j}\,
\widehat{\psi}^{(2)}\bigl(\xi A_{(2)}^{-j}B_{(2)}^{-\ell}\bigr)\,
e^{-2\pi i\langle \xi A_{(2)}^{-j}B_{(2)}^{-\ell},\,k\rangle}.
$$

\item The functions $\psi^{(1,b)}$ and $\psi^{(2,b)}$ are \emph{feasible boundary generators} if $\widehat{\psi}^{(1,b)} (0,0) = \widehat{\psi}^{(2,b)}(0,0)=0$ and if they satisfy for all $\xi\in\mathbb R^2 \setminus \{(0,0)\}$,
$$
|\widehat{\psi}^{(1,b)}(\xi_1,\xi_2)|
\le
\min\{1,|q\xi_1|^{\alpha}\}\,
\min\{1,|q'\xi_1|^{-\gamma}\}\,
\min\{1,|r\xi_2|^{-\gamma}\},
$$
and
$$
|\widehat{\psi}^{(2,b)}(\xi_1,\xi_2)|
\le
\min\{1,|q\xi_2|^{\alpha}\}\,
\min\{1,|q'\xi_2|^{-\gamma}\}\,
\min\{1,|r\xi_1|^{-\gamma}\}.
$$
For $j\ge 0$ and $k\in\mathbb Z^2$, define the boundary shearlets by
$$
(\psi^{(1,b)}_{j,\lceil 2^{j/2} \rceil,k})^{\wedge}(\xi)
=
2^{-\frac{3}{4}j}\,
\widehat{\psi}^{(1,b)}\bigl(\xi A_{(1)}^{-j}B_{(1)}^{-\lceil 2^{j/2} \rceil}\bigr)\,
e^{-2\pi i\langle \xi A_{(1)}^{-j}B_{(1)}^{-\lceil 2^{j/2} \rceil},\,k\rangle},
$$
and
$$
(\psi^{(2,b)}_{j,-\lceil 2^{j/2} \rceil,k})^{\wedge}(\xi)
=
2^{-\frac{3}{4}j}\,
\widehat{\psi}^{(2,b)}\bigl(\xi A_{(2)}^{-j}B_{(2)}^{\lceil 2^{j/2} \rceil}\bigr)\,
e^{-2\pi i\langle \xi A_{(2)}^{-j}B_{(2)}^{\lceil 2^{j/2} \rceil},\,k\rangle}.
$$
\end{enumerate}
\end{definition}

There exist cone-adapted shearlet systems that are compactly supported in space \cite{Kutyniok2012-2} and compactly supported in frequency \cite{Guo2013}. The decay conditions we use in \Cref{coneAdaptedGenerators} as adapted from \cite[Definition 2.6]{Kutyniok2012-2} describe the frequency decay properties that are used to define suitable generators of compactly supported shearlets \cite{Kutyniok2012-2}. We do not reproduce a full construction of the corresponding shearlet systems here. Instead, we recall the two standard classes of cone-adapted shearlet frames which are compactly supported cone-adapted shearlet frames and bandlimited shearlet Parseval frames. It should be noted that the systems appearing in the following results are cone-adapted shearlet systems, but unlike \Cref{shearletsInSpace} they are defined in the frequency-side instead of space. Meaning, denoted by $\mathcal{SH}(c; \Phi, \psi^{(1)},  \psi^{(2)})$, the cone-adapted systems are obtained simply by taking the union of the coarse-scale shearlets and the interior and boundary shearlets corresponding to the horizontal and vertical cones.

\begin{theorem}[{\cite[Theorem 4.9]{Kutyniok2012-2}}] \label{thm:compact-shearlets}
There exist compactly supported generators $\Phi, \psi^{(1)}, \psi^{(2)} \in L^2(\mathbb{R}^2)$ and sampling vector $c \in (\mathbb{R}^+)^2$ such that the associated cone-adapted shearlet system $\mathcal{SH}(c; \Phi, \psi^{(1)}, \psi^{(2)})$ forms a frame for $L^2(\mathbb{R}^2)$.
\end{theorem}

Such frames of compactly supported shearlets are useful when spatial localization is important, but they are not necessarily tight frames. This limitation leads us to the next well-known system of shearlets.

\begin{theorem}[{\cite[Theorem 5]{Guo2013}}]\label{thm:bandLimitedFrame}
There exists a cone-adapted regular discrete system of bandlimited shearlets that forms a Parseval frame for $L^2(\mathbb{R}^2)$. Moreover, the elements of this system are $C^\infty$.
\end{theorem}

\Cref{thm:bandLimitedFrame} provides a Parseval frame example that is used when tightness of the shearlet sparsifying frame is needed. We now state the corresponding nonlinear approximation results for compactly supported and bandlimited shearlets precisely. First, as in \cite{Donoho1999, CandesDonoho2004} we define $STAR^2(A)$ as a class of star-shaped sets with $C^2$-boundaries. 

\begin{definition}
     \begin{itemize}
        \item Let $0 \leq \rho_0 < 1$ and  $\rho \colon \R \to [0,\rho_0]$ be of class $C^2$ and $2\pi$-periodic. Define
         \begin{equation}
           B=\{(\rho\cos\theta,\rho\sin\theta)\in\R^2: \theta\in [0,2\pi],  0\le\rho \le \rho(\theta)\}. \label{eq:star1}
         \end{equation}
         We assume that
        \begin{equation}
        \label{eq:star2}
        \sup |\rho''(\theta)| \le A
        \end{equation}
        for some $A>0$.
        We say that a set $B$ belongs to $STAR^2(A)$ if $B \subseteq  [0,1]^2$ and $B$ is a translate of a set satisfying \eqref{eq:star1} and \eqref{eq:star2}.
    
        \item We denote  the collection of twice differentiable functions supported inside $[0,1]^2$ by $C_0^2([0,1]^2)$. We define the set $\mathcal{E}^2(A)$ of \emph{cartoon-like images} as the collection of functions of the form
        $$
        f = f_0 + f_1 \chi_B,
        $$
        where $f_0,f_1 \in C_0^2([0,1]^2)$, $B \in STAR^2(A)$ and $\|f_0\|_{C^2} ,\|f_1\|_{C^2} \le 1$.
    \end{itemize}
    
\end{definition}

\begin{definition}\label{def:nonlinear-approx}
    For a frame $\{\psi_i\}_{i \in I}$, with its canonical dual frame denoted by $\{\widetilde{\psi}_i\}_{i \in I}$, let $f_N$ be a (not necessarily unique) \emph{$N$-term approximation} of $f$ obtained from the $N$ largest coefficients of its frame expansion, namely,
    $$f_N = \sum_{i \in I_N} \langle f, \psi_i \rangle \widetilde{\psi}_i,$$
    where $I_N \subseteq I$ is the set of indices corresponding to the $N$ largest entries of the coefficient sequence $\{|\langle f, \psi_i \rangle| : i \in I\}$.
\end{definition}

\begin{theorem}[{\cite[Theorem 1.2]{Guo2007}}]
Let $\mathcal{SH}(c; \Phi, \psi^{(1)}, \psi^{(2)})$ be the Parseval frame of bandlimited shearlets as defined in \cite[Section 1.2]{Guo2007}. Let $f \in \mathcal{E}^2(A)$ and $f_N^S$ be the $N$-term shearlet approximation of $f$. Then there exists some $C > 0$ independent of $N$ such that
$$\|f - f_N^S\|_{L^2(\mathbb{R}^2)}^2 \leq C N^{-2} (\log N)^3. $$
\end{theorem}

We also recall that on cartoon-like images, compactly supported shearlet frames have the same near-optimal nonlinear approximation rate as bandlimited ones. 

\begin{theorem}[{\cite[Theorem 1.4]{Kutyniok2011}}]
Let $\mathcal{SH}(c;\Phi,\psi^{(1)},\psi^{(2)})$ be a cone-adapted discrete shearlet frame for $L^2(\mathbb{R}^2)$ with compactly supported generators such that its generators satisfy suitable moment and frequency decay conditions. Let $f \in \mathcal{E}^2(A)$ and $f_N^S$ be an $N$-term shearlet approximation of $f$. Then there exists a constant $C > 0$,  such that
$$\|f - f_N^S\|_{L^2(\mathbb{R}^2)}^2 \leq C N^{-2}(\log N)^3.$$
\end{theorem}

\section{Lower frame bounds for finite frames of one-dimensional wavelets}

In this section, we present some known results on lower frame bounds for a finite wavelet frame.

\begin{definition}
A sequence $\{a_j\}_j$ of real numbers is separated by $\delta > 0$ if
$$|a_j - a_k| \geq \delta \quad \text{for} \quad j \neq k.$$
If $\{a_j\}_j$ is a sequence of positive real numbers, it is called logarithmically separated by $a > 1$ if the sequence $\{\log a_j\}_j$ is separated by $\log a$.
\end{definition}

\begin{remark}
    For convenience, the authors of the following two results assume that the finite wavelet system is parameterized on an ``irregular lattice'' \cite[Section 2]{Christensen2001Wavelet}, as the results are formulated for the scale and translation parameters $\{(a_j, \lambda_n)\}_{j=1,n=1}^{j_0,N}$ and
    $$\psi_{j,n} = \frac{1}{a_j^{1/2}} \psi\bigg(\frac{x}{a_j} - \lambda_n\bigg), \quad j=1,\dots, j_0, \quad n = 1,\dots, N.$$ 
    
    The case of arbitrary parameters $\{(a_\gamma, \lambda_\gamma)\}_{\gamma \in \Gamma}$ can be treated by extending $\{(a_\gamma, \lambda_\gamma)\}_{\gamma \in \Gamma}$ to an irregular lattice, after which the estimates can be used directly. 
\end{remark}

\begin{theorem}[{\cite[Theorem 2.1]{Christensen2001Wavelet}}]
Suppose that $\{a_j\}_{j=1}^{j_0} \subseteq \mathbb{R}_+$ is logarithmically separated by $a > 1$ and that $\displaystyle\sup_{j,k = 1,\ldots,j_0} \frac{a_j}{a_k} \leq K$ for some $K > 1$ and let $\{\lambda_n\}_{n=1}^N \subseteq \mathbb{R}$ be separated. Let $\psi \in L^2(\mathbb{R})$ and suppose there is a positive number $c$ and a non-degenerate interval $I \subseteq [c, \infty)$ such that for any positive number $r$ there are numbers $d_1(r) > 0$, $d_2(r) \geq 0$ and $s(r) > -c$ such that
$$
|\hat{\psi}(x)| \geq d_1(r) \quad \forall x \in I + s(r),
$$
$$
|\hat{\psi}(x)| \leq d_2(r) \quad \forall x > (c + s(r))a,
$$
$$
\frac{d_2(r)}{d_1(r)} \leq r.
$$

Let $A$ be a lower frame bound for $\{e^{-2\pi i \lambda_n x}\}_{n=1}^N$ in $L^2(I)$ and $B$ be an upper bound for it in $L^2(KI)$. Denote by $B'$ an upper bound for $\{e^{-2\pi i \lambda_n x} \hat{\psi}(x)\}_{n=1}^N$ in $L^2(\mathbb{R})$. Let $A_1, \ldots, A_{j_0}$ be a sequence of numbers satisfying
$$
0 < A_1 \leq d_1^2(r)A \text{ for some } r > 0,
$$
$$
0 < A_k \leq \frac{d_1^2(r) AA_{k-1}}{16(\sqrt{r_k} + 1)^2B'} \text{ for some } r_k \in \bigg(0, \frac{AA_{k-1}}{16BB'(k-1)}\bigg], \quad (k \geq 2).
$$
Then $\{\psi_{j,n}\}_{j=1,n=1}^{j_0,N}$ is linearly independent in $L^2(\mathbb{R})$ with lower frame bound $A_{j_0}$.
\end{theorem}

\begin{remark}
Lower bounds $A$ for frames of exponentials $\{e^{-2\pi i \lambda_n x}\}_{n=1}^N$ in $L^2(I)$ have been found in \cite[Proposition 2.3]{Christensen2001Exponential}. They are exponentially small in the number of translation parameters $N$, which makes these bounds unsuitable for our purposes. It should be noted that these lower bounds are provided for not necessarily uniformly distributed translation parameters $\lambda_i$'s, unlike our uniform translates $\lambda_i$. 

Also, if $\{\psi_{j,n}\}_{j=1,n=1}^{j_0,N}$ is a subfamily of a frame $\{\psi_{j,n}\}_{j \in \mathbb{N},n \in \mathbb{Z}}$ with upper frame bound $C$ then we can use $B' = C$, which is a bound independent of $N$. 
\end{remark}

The following theorem gives a simpler set of conditions for bandlimited wavelets. This case is especially relevant here because the known directional wavelets with good nonlinear approximation rates like shearlets and curvelets that form tight frames are bandlimited (see \cite{Guo2007, CandesDonoho2004} and Appendix~\ref{shearletAppendix}).

\begin{theorem}[{\cite[Theorem 2.4]{Christensen2001Wavelet}}] \label{waveletLowerBound}
Let $a_1, \ldots, a_{j_0}$ be a finite sequence of positive numbers logarithmically separated by some $a > 1$ and let $\lambda_1, \ldots, \lambda_N$ be a finite sequence of separated real numbers. Let $\psi \in L^2(\mathbb{R})$ and suppose that $\supp(\hat{\psi}) \subseteq (-\infty, p]$ for some $p > 0$ and that there is a non-degenerate interval $I \subseteq [p/a, p]$ and a positive number $d$ such that
$$
|\hat{\psi}(x)| \geq d \quad \forall x \in I.
$$
Denote a lower bound for $\{e^{-2\pi i \lambda_n x}\}_{n=1}^N$ in $L^2(I)$ by $A$ and an upper bound for $\{e^{-2\pi i \lambda_n x} \hat{\psi}(x)\}_{n=1}^N$ in $L^2(\mathbb{R})$ by $B'$. Then $\{\psi_{j,n}\}_{j=1,n=1}^{j_0,N}$ is linearly independent with lower frame/Riesz bound
$$
A_{j_0} = d^2A\bigg(\frac{d^2A}{16B'}\bigg)^{j_0-1}.
$$
\end{theorem}

\printbibliography

\end{document}